\documentclass[a4paper,twoside,11pt]{article}
\usepackage[english]{babel}
\usepackage[latin1]{inputenc}
\usepackage{indentfirst}
\usepackage{amssymb}
\usepackage{amsmath}
\usepackage{amsthm}

\Large
\theoremstyle{plain}
 \newtheoremstyle{miestilo}{12pt}{\topsep}{\itshape}{}{\bf}{}{ }{}
  \swapnumbers
 \theoremstyle{miestilo}
\newtheorem{teorema}[subsection]{Theorem.}
\newtheorem{proposicion}[subsection]{Proposition.}
\newtheorem{lema}[subsection]{Lemma.}
\newtheorem{corolario}[subsection]{Corollary.}

 \newtheoremstyle{misnotas}{12pt}{12pt}{}{}{\bf}{}{ }{\remark}
  \theoremstyle{misnotas}
 \newtheorem{remark}[subsection]{\ {\bf Remark.}}
  \newtheorem{apartado}[subsection]{\ {\ }}
 \allowdisplaybreaks[2]

\begin{document}

\flushbottom
%
%

{\LARGE  {\bf \centerline{ Local orders in Jordan algebras}} } %

%
%
   %
 \medskip
\medskip
\centerline{ Fernando Montaner \footnote{Partially supported  by
the Spanish Ministerio de Econom\'{\i}a y Competitividad--Fondo Europeo de Desarrollo Regional (FEDER) MTM2013-45588-C3-2-P, and by the
Diputaci\'on General de Arag\'on-Fondo Social Europeo (Grupo de Investigaci\'{o}n de
\'{A}lgebra).}} \centerline{{\sl Departamento de Matem\'{a}ticas,
Universidad de Zaragoza}} \centerline{{\sl 50009 Zaragoza,
Spain}} \centerline{E-mail: fmontane@unizar.es} \centerline{and}
\centerline{ Irene Paniello \footnote{Partially supported  by
the Spanish Ministerio de Econom\'{\i}a y Competitividad--Fondo Europeo de Desarrollo Regional (FEDER) MTM2013-45588-C3-2-P.}} \centerline{{\sl
Departamento de Estad\'{\i}stica e Investigaci\'{o}n Operativa, Universidad
P\'{u}blica de Navarra}}
 \centerline{{\sl 31006 Pamplona, Spain}}
\centerline{E-mail: irene.paniello@unavarra.es}
%
%
\begin{abstract}
We study a notion of order in Jordan algebras based on the version for Jordan algebras of the ideas of Fountain and Gould \cite{fogo1} as adapted to the Jordan context by Fern\'{a}ndez-L\'{o}pez and Garc\'{\i}a-Rus \cite{fg1}, making use of results on general algebras of quotients of Jordan algebras. In particular, we characterize the set of Lesieur-Croisot elements of a nondegenerate Jordan algebra as those elements of the Jordan algebra lying in the socle of its maximal algebra of quotients, and apply this relationship to extend to quadratic Jordan algebras the results of Fern\'{a}ndez-L\'{o}pez and Garc\'{\i}a-Rus on local orders in nondegenerate Jordan algebras satisfying the descending chain condition  on principal inner ideals and not containing ideals which are nonartinian quadratic factors.
\end{abstract}

\section*{Introduction}

Local orders for Jordan algebras were introduced and studied by Fern\'{a}ndez-L\'{o}pez and Garc\'{\i}a-Rus in \cite{fg1, fg2} inspired by the work of Fountain and Gould \cite{fogo1, fogo2, fogo3}, and \'{A}hn and M\'{a}rki \cite{anhm1, anhm2} on orders of associative  rings following ideas on quotients in semigroup theory. The original purpose of that research on associative algebras was to introduce a notion of localization inspired in Ore's construction but without the requirement of having an identity element, so that the "regular elements" were not intended to become invertible, but merely "locally invertible". That turn to locality fits well with some of the ideas of Jordan theory, in particular with that of local algebra at a given element \cite{am-local}, a remark that made natural the step taken by Fern\'{a}ndez-L\'{o}pez and Garc\'{\i}a-Rus of extending these notion to the Jordan context.

Fern\'{a}ndez-L\'{o}pez and Garc\'{\i}a-Rus' work was preceded and inspired by results on Jordan algebras of fractions which originated in the question raised by Jacobson on  whether results similar to Ore's construction could be obtained for Jordan algebras  \cite[p. 426]{jac-struc-rep}.  As it is well known, Ore's results were extended by Goldie \cite{goldie1, goldie2} to the study of embeddability of associative rings in simple or semisimple artinian rings, which, in turn, motivated associative localization theory. In the realm of Jordan theory, Jacobson's question, or rather the related question on the possibility  of extending  Goldie's results to the Jordan setting was first answered by Zelmanov  \cite{z-goldie-1,z-goldie-2} (later extended  by Fern\'andez-L\'opez, Garc{\'\i}a-Rus and Montaner to quadratic Jordan algebras \cite{fgm}). As for Jacobson's original question, a complete answer was given by Mart\'{\i}nez \cite{martinez} based on a quite different approach that allowed her to provide necessary and sufficient Ore-like conditions for the existence of algebras of fractions of Jordan algebras (with $\frac{1}{6}\in \Phi$) (later generalized to quadratic algebras by Bowling and McCrimmon \cite{bomc}.)

Those results opened the way to a sizable area of research on algebras of quotients of Jordan algebras. We refer to  \cite{densos} for a concise overview of that field, some of whose results will be recalled in the following sections as needed. At this point we limit ourselves  to mentioning two notions which were introduced in those  developments, and which we will need in order to describe some of the problems to whose solution this paper is devoted.

First of all, whereas Mart{\'\i}nez's answer of Jacobson's problem faithfully parallels the associative situation (although through a quite different proof),  Fern\'andez-L\'{o}pez, Garc{\'\i}a-Rus and Montaner showed in \cite{fgm} that the Jordan version of Goldie theory deviates from its associative counterpart  at a significant point that will be pivotal in the present research: the natural Jordan version of the characterization of left Goldie associative algebras (and similarly, right Goldie algebras) as those for which left (resp. right) essential inner ideals are precisely the ones that contain regular elements, namely the characterization of Goldie Jordan algebras as those for which an
inner ideal is essential if and only if it contains injective
elements, no longer holds for Jordan algebras. However, that missing Goldie-type property has its own interest since, as proved in \cite{fgm}, algebras that satisfy it are precisely those which are orders in nondegenerate algebras of finite capacity, which is the natural finiteness condition from the viewpoint of the classical Jordan theory based on the use of idempotents. Those algebras were termed Lesieur-Croisot algebras (LC-algebras for short) in \cite{lc}. Jordan algebras having local algebras that are Lesieur-Croisot  were studied later by Montaner and Toc\'{o}n in \cite{lc, lc-pr}.

A second notion belonging to the just mentioned study of algebras of quotients, and that will be central in the present research, is the formulation in the Jordan setting of a more general construction  of algebras of quotients that parallels Lambeck-Utumi's associative algebras of quotients, and  that turns out to provide a common environment for most of the previously developed constructions of quotients for nondegenerate Jordan algebras (see \cite{densos}).

In spite of its generality, and in contrast to the situation in the associative theory \cite{anhm1}, local orders as defined by Fern\'andez-L\'opez and Garc{\'\i}a-Rus have not been shown to fit in that setting. Since the study of that notion of local orders is the objective of the present paper, that will be one of the issues we will address, although we will restrict that study to the case in which the algebras are local orders in over-algebras with dcc on principal inner ideals.

The organization of the paper, and the questions we will deal with will be the following:

After an initial section of preliminaries, in section 1 we recall basic facts on the just mentioned two classes of algebras of quotients of Jordan algebras, classical algebras of fractions, including the Jordan analogues of the notions and results of Goldie Theory, and general algebras of quotients which are Jordan analogues of Lambek-Utumi's algebras of quotients in the associative theory.

In addition to recalling the basic notions of these two kinds of algebras of quotients, including that of  LC-elements of a Jordan algebra, we prove a central result of the section which will ease computations later on, namely the coincidence of essentiality and density for inner ideals in algebras in which every element is LC.

Sections 2 and 3 are devoted to a deeper study of nondegenerate Jordan algebras having nonzero LC-elements, addressing our first objective, the study of the existing connection between the set of LC-elements of a nondegenerate Jordan algebra and the socle of its maximal algebra of quotients. Drawing inspiration from a well known result of associative theory,
we prove in section 2 that the set of LC-elements of a strongly prime Jordan algebra coincides with the intersection of the socle of its maximal algebra of quotients with the original algebra, and extend that result to nondegenerate algebras in Section 3.

In section 4 we introduce the notion of local order of a Jordan algebra to which this paper is devoted. This notion basically coincides with the one introduced by Fern\'{a}ndez-L\'{o}pez and Garc\'{\i}a-Rus, although it slightly differs from theirs
 in that it makes use of the quadratic version \cite{fgm} of Goldie's theorems for Jordan algebras, rather that its original linear form \cite{z-goldie-1, z-goldie-2}.

Since, as mentioned in \cite{fg1}, the original motivation for the introduction of local orders in associative theory was the introduction of an Ore-like localization in algebras which need not have an identity element thus generalizing Goldie's theorems for Jordan algebras, the regularity condition to be satisfied by the resulting over-algebra is a natural generalization of the artinian property to nonunital algebras, namely having the dcc on principal inner ideals, that is, being equal to its socle. According to that, we give here a version of the definition of local order in algebras that equal their socles, and prove that the result of the previous section on the socle of the algebra of quotients of a Jordan algebra is properly understood through the notion of local order, since the LC-ideal of a nondegenerate algebra turns out to be a local order in the socle of its algebra of quotients. As a final result in this section, we prove an analogue of a fact proved by \'{A}nh and M\'{a}rki  \cite{anhm1}, and show that local orders in algebras that equal their socles are indeed orders in the sense of \cite{densos} thus showing that this kind of quotients can also be viewed through the framework of that construction.

In Section 5 we revisit the theory of Fern\'{a}ndez-L\'{o}pez and Garc\'{\i}a-Rus developed in 	\cite{fg1,fg2}, on what they named local Goldie conditions, and their consequences. We add to that study the local LC-condition, and obtain here all those results as a natural application of the theory developed in the previous sections, thus obtaining the characterization of strongly prime and nondegenerate Jordan algebras that are local orders either in local artinian algebras, or in algebras satisfying the dcc on principal inner ideals.

%
%

\addtocounter{section}{-1}

\section{ Preliminaries}

\begin{apartado}
  We will work with    Jordan  algebras over a unital commutative ring of scalars $\Phi$
   which will be fixed throughout. We refer to
\cite{jac-struc,mcz} for   notation, terminology and basic
results on Jordan algebras. We will occasionally make use some results on  Jordan pairs, mainly
obtained from algebras, for which we refer to \cite{loos-jp}, and we will often rely on an associative background, both as an ingredient of Jordan theory when dealing with special Jordan algebras, and as a source of notions which have been extended to the Jordan algebra setting, among which we will mainly consider those  from localization theory, for which we refer to \cite{st}.

We will make use of the identities proved in
\cite{jac-struc}, which will be quoted with the
labellings QJn. In this section we recall some
of those basic results  and notations together with some other results   that will
be used in the paper.
\end{apartado}

\begin{apartado}A Jordan algebra has products   $U_xy$ and $x^2$, quadratic in $x$
and linear in $y$, whose linearizations are
 $U_{x,z}y=V_{x,y}z=\{x,y,z\}=U_{x+z}y-U_xy-U_zy$ and $x\circ y=V_xy=(x+y)^2-x^2-y^2$.

 We will denote by $\widehat{J} $ the free unital hull $\widehat{J}=\Phi 1\oplus
 J$ with products $U_{\alpha
 1+x}(\beta1+y)=\alpha^2\beta1+\alpha^2y+\alpha x\circ
 y+2\alpha\beta x+\beta x^2+U_xy$ and
 $(\alpha1+x)^2=\alpha^21+2\alpha x +x^2$.

It is well known that  any associative algebra $A$ gives rise to a Jordan algebra $A^+$ with products $U_xy=xyx$ and $x^2=xx$. A Jordan algebra is special if it is isomorphic to a subalgebra of an algebra $A^+$ for an associative $A$. If $A$ has an involution $\ast$ then $H(A,\ast)=\{a\in A\mid a=a^\ast\}$ is a Jordan subalgebra of $A^+$ and so are ample subspaces $H_0(A,\ast)$ of symmetric elements of $A$, subspaces such that $a+a^\ast, aa^\ast$ and $aha^\ast$ are in $H_0(A,\ast)$ for all $a\in A $ and all $h\in H(A,\ast)$.
\end{apartado}

\begin{apartado} A $\Phi$-submodule $K$ of a Jordan algebra $J$ is
an \emph{inner ideal} if $U_k\widehat{J}\subseteq K$ for all $k\in
K$. An inner ideal $I\subseteq J$ is an \emph{ideal} if
$\{I,J,\widehat{J}\}+U_JI\subseteq I$, a fact that we will denote in the usual way $I\triangleleft J$. If $I$, $L$ are ideals of
$J$, so is their product $U_IL$  and in particular   the
\emph{derived ideal} $I^{(1)}=U_II$ of $I$. An (inner) ideal  of $J$ is
\emph{essential}\  if it has nonzero intersection  with any
nonzero (inner) ideal of $J$.

A Jordan algebra $J$ is \emph{nondegenerate} if $U_xJ\neq0$ for any nonzero $x\in J$, and \emph{prime} if $U_IL\neq0$ for any nonzero ideals $I$ and $L$ of $J$. The algebra $J$ is said to be \emph{strongly prime} if $J$ is both nondegerate and prime.
\end{apartado}

\begin{apartado}\label{annihilator_def}
If $X\subseteq J$ is a subset of a Jordan algebra $J$, the
\emph{annihilator} of $X$ in $J$ is the set $ann_J(X)$ of all
$z\in J$ such that $U_xz=U_zx=0$ and
$U_xU_z\widehat{J}=U_zU_x\widehat{J}=V_{x,z}\widehat{J}=V_{z,x}\widehat{J}=0$
for all $x\in X$. The annihilator  is always an inner ideal of $J$,
and it is an ideal if $X$ is an ideal. If $J$ is a nondegenerate
Jordan algebra  and $I$ is an ideal of $J$, then the annihilator
of $I$ in $J$ can be characterized as follows
(\cite{mc-inh,pi-ii}):
$$ann_J(I)=\{z\in J\ \mid\ U_zI=0\}=\{z\in J\ \mid\ U_Iz=0\}.$$
\end{apartado}

\begin{apartado}\label{local algebras} For any element  $a$ in a Jordan algebra $J$,
   the   \emph{local algebra  $J_a$
of $J$ at $a$} is the quotient of the $a$-homotope $J^{(a)}$, defined over the $\Phi$-module $J$ with operations $U^{(a)}_xy = U_xU_ay$ and  $x^{(2,a)}= U_xa$, by the
ideal $Ker\,  a$ of $J^{(a)}$ of all the elements $x\in J$ such that
$U_ax=U_aU_xa=0$. If $J$ is nondegenerate  the above conditions on $x$
reduce  to $U_ax=0$.
Local algebras at nonzero elements of a nondegenerate (resp. strongly prime) Jordan algebra are nondegenerate (resp. strongly prime) \cite[Theorem 4.1]{acm}. (We recall that similar definitions can be given for associative algebras, for which we will also use the notation $R_x$ for the local algebra at an element $x$ of   $R$.)
\end{apartado}

\begin{apartado}\label{def-double}A Jordan algebra or triple system $J$ gives rise to its double Jordan pair $V(J) = (J,J)$, with (quadratic) operations obtained from the quadratic operation of $J$: $Q_xy = U_xy$ or $P_xy$. Reciprocally, every Jordan pair $V = (V^+,V^-)$ gives rise to a (polarized) triple system $T(V) = V^+\oplus V^-$. If $J$ is a triple system, it is obvious that if $I\triangleleft J$, then $(I,I)$ is an ideal of $V(J)$, however, not all ideals of $V(J)$ arise in that  way from ideals of $J$. In fact, if $I = (I^+,I^-)$ is an ideal of
$V(J) = (J,J)$ we may well have $I^+\ne I^-$, and even $I^+\cap I^-=0$, so that $I^+\oplus I^-$ is a polarized ideal of $J$ as a triple system. We however have the following result:
\end{apartado}

\begin{lema}\label{double ideals}Let $J$ be a nondegenerate Jordan algebra, if $I = I^+\oplus I^-$ is a polarized ideal of $J$ as a triple system (that is $(I^+,I^-)$ is an ideal of $V(J)$ and $I^+\cap I^- =0$), then $I=0$.
\end{lema}
\begin{proof}Following the proof of \cite[Proposition 2.4]{pi-ii}, the set $I_{alg} = \{y \in  I\mid y^2+y\circ I \subseteq I\}$ is an ideal of $J$ as an algebra which satisfies $U_IJ \subseteq I_{alg} \subseteq I$. In particular, $I_{alg}$ is still polarized, and $I_{alg} = 0$ if and only if $I =0$. Therefore we can assume that $I=I^+\oplus I^-$ is a polarized algebra ideal.

Take $x^\sigma \in I^\sigma$ for $\sigma = \pm$. Since $U_{I^\sigma}U_{I^\sigma}J = 0$, we have $U_{(x^\sigma)^2}J= U_{x^\sigma}^2J =0$, hence $(x^\sigma)^2 =0$ by nondegeneracy of $J$. Denoting $a = x^+\circ x^-$, for any $z\in J$ we have $U_za = U_z(x^+\circ x^-) = z\circ \{x^+,x^-,z\} + \{z^2,x^-,x^+\} \in J\circ \{I^+,I^-,J\} + \{J, I^-,I^+\} =0$, since $\{I^\sigma,I^{-\sigma},J\} \subseteq I^\sigma\cap I^{-\sigma} = 0$. Then $a \in ann_J(I)$ by the characterization of annihilators of ideals  mentioned in \ref{annihilator_def}. Since $I$ is an algebra ideal, $a\in I$, hence $a \in I\cap ann_J(I) =0$ by nondegeneracy of $J$. As a consequence of these equalities, we get $x^2= (x^++x^-)^2 = (x^+)^2+(x^-)^2+x^+\circ x^- = 0$, hence $I^2=0$.

Consider now a tight unital hull $J'$ of $J$ (a unital hull $J\triangleleft J' = J+\Phi 1$ which is tight over $J$, and therefore inherits nondegeneracy from $J$ \cite[0.16,0.17]{mcz}). Since $I$ is an ideal of $J$, it is also an ideal of $J'$, and the equality $I^2=0$ can be rewritten in $J'$ as $U_I1=0$, which again by the characterization of the annihilator mentioned in \ref{annihilator_def}, implies $1 \in ann_{J'}(I)$, and thus $I = U_1I =0$.
\end{proof}

\begin{apartado}\label{socle-def}The \emph{socle} $Soc(J)$ of a nondegenerate Jordan algebra $J$ is the sum of all minimal inner ideals of $J$. The socle of linear Jordan algebras has been studied by Osborn and Racine in \cite{orsocle}. For Jordan pairs over arbitrary rings of scalars, the socle has been thoroughly studied by Loos in \cite{loos-socle}. Our handling of the socle will rely on that reference. In particular, it is proved there that the socle of a nondegenerate Jordan pair $V$ is a direct sum of simple ideals of $V$, consists of regular elements, and satisfies the dcc on principal inner ideals \cite{loos-finiteness}.
Applying these assertions to the equality $V(Soc(J)) = Soc(V(J))$ for a Jordan algebra $J$, gives that $Soc(J)$
consists of regular elements and satisfies the dcc on principal inner ideals. On the other hand, if $I = (I^+,I^-)$ is a simple ideal of $V(J)$, then $V(I^+\cap I^-) \subseteq I$ is an ideal of $V(Soc(J))$, hence either $I = V(I^+\cap I^-)$, which gives $I ^+=I^-$, and $I= V(L)$ for the simple ideal $L = I^+=I^-$, or $I^+\cap I^- =0$. The latter case means that $I^++I^-$ is a polarized ideal of $J$, so it is the zero ideal by \ref{double ideals}. As a consequence, $Soc(J)$ is a direct sum of simple ideals coinciding with their own socles.

The elements of the socle of a nondegenerate Jordan algebra $J$ are exactly those whose local algebra $J_x$ has finite capacity \cite[Lemma 0.7(b)]{pi-i}. From \cite{loos-socle} we also obtain that a nondegenerate Jordan algebra $J$ satisfying the dcc on principal inner ideals (equivalently coinciding with its socle) also satisfies acc on the inner ideals $ann_J(x)$ for $x\in J$.
  \end{apartado}

 \begin{apartado}\label{locally artinian} Let $J$ be a Jordan algebra. We will say that $J$ is \emph{locally artinian} if  for any $x\in J$, $J_x$ is artinian. Note that if $J$ is locally artinian, then each local algebra $J_x$ has finite capacity, hence $J = Soc(J)$ by \ref{socle-def} (obviously, the analogous definition obtained for nondegenerate algebras by substituting  'artinian' by 'of finite capacity' in the definition of locally artinian gives nothing new, since as mentioned before, those are just the algebras that equal their socle).

It follows from the structure of inner ideals of Jordan
algebras having finite capacity \cite{mc-in} that simple Jordan
algebras with finite capacity (equivalently with dcc on principal
inner ideals) are either artinian or the Jordan algebra of a
nondegenerate quadratic form containing an infinite dimensional
totally isotropic vector subspace, in short, a \emph{nonartinian
quadratic factor}. Every element in a simple Jordan algebra with
dcc on principal inner ideals which is not a nonartinian quadratic
factor has finite uniform dimension (in the sense of \cite[p. 425]{fgm}, which we will explicitely recall later). Moreover such Jordan algebras
are locally artinian, hence in particular, for each idempotent $e$ in
$J$, the unital Jordan algebra $U_eJ = J_2(e) \cong J_e$ (see \cite[Example 1.12]{Subquot}) is  simple and artinian.
\end{apartado}

\begin{lema}\label{estructure loc artinian}For a nondegenerate Jordan algebra, the following facts are equivalent:
\begin{enumerate}
\item[(i)] $J$ is locally artinian,
\item[(ii)]$J = Soc(J)$, and $J$ does not contain ideals that are nonartinian quadratic factors.
\item[(iii)]$J$ is a direct sum of simple Jordan algebras coinciding with their socles, none of which is a nonartinian quadratic factor.

\end{enumerate}

\end{lema}
\begin{proof}(i)$\Rightarrow$(ii) As noted before, for an algebra $J$, being locally artinian implies $J = Soc(J)$. Moreover, if $J$ contains an ideal $I$ which is a nonartinian quadratic factor, then $I$ is unital, and the local algebra of $J$ at the unit element $e$ of $I$ is the nonartinian algebra $J_e\cong I$.

(ii)$\Rightarrow$(iii) That $J$ is a direct sum of simple algebras coinciding with their socles is a consequence of the general result on the structure of the socle of a nondegenerate algebra \ref{socle-def}. That none of these summands is a nonartinian quadratic factor stems directly from the fact that local algebras of the summands are local algebras of $J$ itself, and local algebras at the unit element of a nonartinian quadratic factor are themselves nonartinian quadratic factors.

(iii)$\Rightarrow$(i) Local algebras of direct sums of simple Jordan algebras coinciding with their socles are Jordan algebras with finite capacity, so either they are artinian or they are nonartinian quadratic factors. This latter case can occur only if the local algebra is taken at an element of a direct summand which is itself a nonartinian quadratic factor, but this is ruled out by condition (iii).
\end{proof}

\section{Algebras of quotients}

As mentioned in the introduction, the study of algebras of quotients of Jordan algebras draws its inspiration from  associative theory  (see \cite{densos}). We recall next some basic notation from the latter and refer the reader to \cite{anillos,st} for basic results about algebras of
 quotients for associative algebras.

\begin{apartado}Let $L$ be a left ideal of an associative algebra $R$. Recall the usual notation
(for instance, see \cite{st})  $(L:a)$, with $a\in R$, for the
 set of all elements $x\in R$ such that $xa\in L$. A left ideal $L$
 of $R$ is \emph{dense} if $(L:a)b\neq0$
 for any $a\in R$ and any nonzero $b \in R$.
\end{apartado}

\begin{apartado} The associative agebras naturally arising in Jordan theory are associative
 envelopes of Jordan algebras, and therefore   carry an involution. That makes important
 to be able to extend involutions to their algebras of quotients. The fact that this is not always possible for the one-sided maximal algebras
 of quotients $Q_{max}^r(R)$ and $Q_{max}^l(R)$ leads to the  use of the maximal symmetric algebra of quotients
 $Q_\sigma(R)$ (see \cite{lanning}) as an adequate
 substitute of that algebra. Recall that $Q_\sigma(R)$  is
 the set of elements $q\in Q_{max}^r(R)$ for which there exists a
 dense left ideal $L$ of $R$ such that $Lq\subseteq R$
 (which up to a canonical isomorphism can be viewed  symmetrically as the set of all $q\in Q_{max}^l(R)$ for which there exists a
 dense right ideal $K$ of $R$ such that $qK\subseteq R$).  If $R$ has
 an involution,   $Q_\sigma(R)$ is the biggest subalgebra of $Q_{max}^r(R)$ and
 $Q_{max}^l(R)$ to which that involution extends.
 \end{apartado}

\begin{apartado}\label{new_ii} Let $J$ be a Jordan algebra, $K$ be  an inner ideal of $J$ and $a\in J$.
Following
  \cite{densos, esenciales} we set $$ (K:a)_L=\{x\in K\  \mid   x\circ a\in K  \},$$
$$ (K:a)=\{x\in K\  \mid U_ax,\ x\circ a\in K  \}.$$
It is straightforward to check that both $ (K:a)_L$ and $ (K:a)$ are  inner ideals of $J$ for
all $a\in J$, and that in addition, the containment $U_{(K:a)_L}K\subseteq (K:a)$   holds \cite[Lemma 1.2]{densos}. Given any  finite family of elements $a_1,\ \ldots, a_n\in
J$, we inductively define $(K:a_1:a_2:\ldots:a_n)=((K:a_1:\ldots
a_{n-1}): a_n)$.\end{apartado}

\begin{apartado}\label{nii} An inner ideal $K$ of $J$ is said to be \emph{dense} if
$U_c(K:a_1:a_2:\ldots:a_n)\neq0$ for any finite collection of
elements  $a_1,\ \ldots, a_n\in J$ and any $0\neq c\in J$. Different
characterizations of density are given in \cite[Proposition 1.9]{densos}. Recall that
if $K$ is a dense inner ideal of $J$ so are the inner ideals  $(K:a)$
for all $a\in J$ \cite[Lemma 1.8]{densos}.
\end{apartado}

\begin{apartado}\label{denominadores} Let $\widetilde{J}$ be a Jordan algebra,  $J$  be  a subalgebra of
$\widetilde{J}$, and   $\widetilde{a}\in\widetilde{J}$. Recall from
\cite{pi-ii} that an element $x\in J$ is a \emph{$J$-denominator} of
$\widetilde{a}$ if the following multiplications take
$\widetilde{a}$
 back into $J$:

 \centerline{ \begin{tabular}{lll}
   (Di)\ $U_x\widetilde{a}$ &  (Dii)\ $U_{\widetilde{a}}x$ &  %
   (Diii)\ $U_{\widetilde{a}}U_x\widehat{J}$ \\
   (Diii')\ $U_xU_{\widetilde{a}}\widehat{J}$  & %
   (Div)\ $V_{x,\widetilde{a}} \widehat{J}$  & (Div')\ $V_{ \widetilde{a},x} \widehat{J}$ \\
 \end{tabular}}

 \noindent We will denote the set of $J$-denominators of $\widetilde{a}$ by
 ${\cal D}_J(\widetilde{a})$. It has been proved in \cite{pi-ii}
 that ${\cal D}_J(\widetilde{a})$ is an inner ideal of $J$. Recall also from  \cite[p. 410]{fgm} that any $x\in J$ satisfying (Di),
 (Dii), (Diii) and (Div) belongs to ${\cal D}_J(\widetilde{a})$. The following procedure for obtaining denominators is given in \cite[Lemma 2.2]{fgm}, and has the advantage of being "context free", that is of not depending of the overalgebra $\widetilde{J}$: for any $x\in J$, the containments $x\circ \widetilde{a}, U_x\widetilde{a}\in J$ imply $x^4\in {\cal D}_J(\widetilde{a})$.
\end{apartado}

\begin{apartado}\label{aq} Let $J$ be a subalgebra of a Jordan algebra $Q$. Following \cite{densos}, we say that $Q$ is a \emph{general algebra of quotients of $J$} if the
following conditions hold:
\begin{enumerate}\item[(AQ1)] ${\cal D}_J(q)$ is a dense
inner ideal of $J$ for all $q\in Q$. \item[(AQ2)] $U_q {\cal
D}_J(q)\neq0$ for any nonzero $q\in
Q$.\end{enumerate}

Note that any nondegenerate Jordan algebra is its own algebra of
quotients. Conversely any Jordan algebra having an algebra of
quotients is nondegenerate.

A different, though closely related approach to algebras of quotients was carried out in \cite{esenciales} by using essential inner ideals as sets of denominators. That second approach, which in fact motivated and inspired the one in \cite{densos} as well as some other treatments of algebras of quotients in Jordan algebras (see the references in \cite{densos}), has the advantage that checking essentiality is significantly simpler than checking density. However, for that choice to be feasible, the additional condition of strong nonsingularity, introduced in \cite{esenciales}, is needed. We will  comment on that at the end of the present section.\end{apartado}

\begin{apartado}\label{def}  An algebra  of quotients $Q$ of a Jordan algebra $J$
is said to be a \emph{maximal  algebra of quotients} of $J$ if for
any other algebra of quotients $Q'$ of $J$ there exists a
homomorphism $\alpha: Q' \to Q$ whose restriction to $J$ is the
identity map.

If there exists, a  maximal algebra of quotients of a nondegenerate Jordan algebra is easily seen to be unique up to an isomorphism fixing the subalgebra $J$. The existence of maximal algebras of quotients
of nondegenerate Jordan algebras was proved in \cite[Theorem 5.8]{densos}. We denote by $Q_{max}(J)$ the maximal algebra of quotients of a nondegenerate Jordan algebra $J$.
\end{apartado}

\begin{teorema} Any nondegenerate Jordan algebra $J$
has a maximal algebra of quotients $Q_{max}(J)$.
\end{teorema}

 We refer to \cite[Theorem  3.11 and Theorem 4.10]{densos} for the explicit
description of  the maximal algebra  of quotients of  a
nondegenerate Jordan algebra.  We also recall the straightforward fact that maximal algebras of quotients of nondegenerate Jordan algebras are unital  \cite[Remark 5.9]{densos}.

As mentioned in the introduction, and as was the case in the associative setting,  the study of general algebras of quotients of Jordan algebras was originated in the study of algebras of fractions. We next recall some of the facts concerning these.

\begin{apartado} A nonempty subset $S\subseteq J$ is a \emph{monad} if $U_st$ and $s^2$ are in $S$ for all $s,t\in S$. A
subalgebra $J$ of a unital Jordan algebra $Q$ is an \emph{$S$-order in $Q$} or equivalently $Q$ is a \emph{$S$-algebra of quotients, or an algebra of fractions (of $J$ relative to $S$)} if:
\begin{enumerate}
\item[(ClQ1)] every element $s\in S$ is invertible in $Q$.
\item[(ClQ2)] each $q\in Q$ has a $J$-denominator in $S$.
\item[(ClQ3)] for all $s,t\in S$, $U_sS\cap U_tS\neq \emptyset$.
\end{enumerate}
\end{apartado}

\begin{apartado} An element $s$ of a Jordan algebra $J$ is said to be
\emph{injective} if the mapping $U_s$ is injective over $J$.
Following \cite{fgm} we will denote by $Inj(J)$ the set of injective
elements of $J$. \end{apartado}

\begin{apartado}\label{classical_order}
A Jordan algebra $Q$ containing $J$ as a subalgebra
is a \emph{classical algebra of quotients of $J$} or an \emph{algebra of fractions} of $J$ (and $J$ is a
\emph{classical order in $Q$}) if all injective elements of $J$ are
invertible in $Q$ and for all $q\in Q$, ${\cal D}_J(q)\cap
Inj(J)\neq\emptyset$.  In other words, classical algebras of quotients are $S$-algebras of quotients (or  algebras of fractions relative to $S$) for $S=Inj(J)$. Moreover, they are general  algebras
of quotients (in the sense of \ref{aq}, as usual) \cite[Examples 2.3.5]{densos}.
\end{apartado}

\begin{apartado}\label{InTEx} The proximity of a nondegenerate algebra and its algebras of quotients can be expressed through the following notion introduced in \cite[p. 414]{fgm}, which includes that of classical algebras of quotients as a particular case: Let $J \leq \widetilde{J}$  be Jordan algebras. An over-algebra  $ \widetilde{J}$ is said to be an \emph{innerly tight extension} of  $J  $  if
\begin{enumerate}
\item[$\bullet$] $U_{\widetilde{a}}J\cap J\neq 0$ for any $0\neq  \widetilde{a}\in  \widetilde{J}$, and
\item[$\bullet$] ${\cal D}_J(\widetilde{a})$ is an essential inner ideal of $J$ for any $\widetilde{a}\in \widetilde{J}$.
\end{enumerate}
By \cite[Lemma 2.4]{densos}, an algebra of quotients of a nondegenerate Jordan algebra is an innerly tight extension. As for the reciprocal, a partial result follows from  \cite[Proposition 2.10]{fgm}:  unital innerly tight extensions with finite capacity are classical algebras of quotients.
\end{apartado}

\begin{apartado}\label{goldie stuff}Let $J$ be a Jordan algebra. We follow \cite{fgm} for the next definitions that will be used below, in the statement of the Goldie theorem for Jordan algebras.

-- For a subset $X\subseteq J$,
  denote by $[X]_{J}$ the inner ideal of $J$ generated by $X$.
 A family $\{K_i\}_{i\in I}$ of nonzero inner ideals of $J$ \emph{forms
 a direct sum} if $K_i\cap [\sum_{j\neq i} K_j]_{J}=0$ for each $i\in I$. Following \cite[p. 426]{fgm}, we say that a Jordan algebra $J$ satisfies the $acc(\oplus)$ if it does not have infinite families of nonzero inner ideals that form a direct sum.  In analogy with the corresponding notion in associative theory (see \cite[p. 361]{anillos} or, under the name of finite right rank, \cite[II.2]{st}),  the
 \emph{uniform (or Goldie) dimension}
 $ udim(J)$ of a Jordan algebra $J$ is defined as the supremum of the $n\geq1$ such that there are $K_1, \ldots, K_n$ nonzero inner ideals of $J$ which form a direct sum (including the possibility that the set of such numbers $n$ is not bounded, in which case $J$ will be said to have \emph{infinite uniform dimension}).

 As for its associative counterpart, and in accordance with the notation used in \cite[Lemma 5.4]{fgm}, for an associative algebra $R$, we will denote respectively by $u_ldim(R)$ and $u_rdim(R)$ the left and right uniform dimensions of $R$. If $x\in R$ we put $u_ldim(x)= u_ldim(R_x)$ and $u_rdim(x) = u_rdim(R_x)$. (Note that $u_ldim(R_x)$ coincides with the uniform dimension of $Rx$ as a left $R$-module, and similarly with the "right" instead of the "left" version.)

If a nondegenerate Jordan algebra $J$ satisfies the $acc(\oplus)$ if and only if it has finite uniform dimension \cite[Proposition 7.6]{fgm}.

--The \emph{singular
  set} of a Jordan algebra $J$ is
 $$\Theta(J)=\{x\in J\mid ann_J(x)\ \hbox{is an  essential inner ideal of $J$}\}.$$

 If $J$ is
 nondegenerate then
 $\Theta(J)$ is an ideal of $J$ \cite[Theorem 6.1]{fgm}, and $J$ is \emph{nonsingular} if $\Theta(J)=0$.

-- A nonzero element $u\in J$ is \emph{uniform} if
$ann_J(u)=ann_J(x)$ for any nonzero $0\neq x\in U_u\widehat{J}$. It is straightforward that
$ann_J(u)\subseteq ann_J(x)$ for every $ x\in U_u\widehat{J}$, hence every nonzero element $u
\in J$ with maximal annihilator is uniform. If $J$ satisfies the acc on annihilators of  its elements, every nonzero inner ideal of $J$ contains a uniform element \cite[p. 55]{fg2}. Uniform elements of nondegenerate Jordan algebras can be characterized through their local algebras. Indeed, by \cite[Proposition 8.4]{fgm},
 a nonzero element of a nondegenerate Jordan algebra is uniform if and only if the local algebra at that element is a Jordan domain.
\end{apartado}

  \begin{apartado}
 A nondegenerate Jordan algebra is \emph{Goldie} if it satisfies the acc on annihilators and has no
 infinite direct sum
 of inner ideals.  Different equivalent characterizations of Goldie Jordan algebras are given in  \cite[Theorem 9.3]{fgm}, among them we select the following:
 \end{apartado}

\begin{teorema}\label{goldieThm} {\rm\cite[Theorem 9.3]{fgm}} For a Jordan algebra $J$ the following conditions are equivalent:
\begin{enumerate}
\item[(i)]  $J$ is a classical order in a nondegenerate artinian Jordan algebra $Q$.
\item[(ii)] $J$ is nondegenerate, satisfies the acc on the annihilators of its elements and has finite uniform dimension.
\item[(iii)] $J$ is nondegenerate, any nonzero ideal of $J$ contains a uniform element, and $J$ has finite uniform dimension.
\item[(iv)] $J$ is nondegenerate, nonsingular and has finite uniform dimension.
\end{enumerate}
Moreover, $Q$ is simple if and only if $J$ is strongly prime.
\end{teorema}

  The study in \cite{fgm} of Jordan algebras  that are classical orders in semisimple artinian algebras (that is nondegenerate Goldie Jordan algebras) extends to the study of a wider class of algebras, those which are classical orders in  nondegenerate unital Jordan algebras  of finite capacity. The main result on those is the following:

\begin{teorema} {\rm\cite[Theorem 10.2]{fgm}} A Jordan algebra $J$ is a classical order
in a nondegenerate unital Jordan algebra $Q$ with finite capacity if
and only if it is nondegenerate and satisfies the following
property: An inner ideal $K$ of $J$ is essential if and only if $K$
contains an injective element. Moreover, $Q$ is simple if and only
if $J$ is prime.\end{teorema}

\begin{apartado}\label{lc def} Following \cite{lc, lc-pr}, a Jordan algebra $J$ satisfying the above equivalent properties
will be called a  \emph{Lesieur-Croisot Jordan algebra} or an  \emph{LC Jordan algebra}, for short.

 The set $LC(J)$ of elements $x\in J$ at which
 the local algebra $J_x$ is LC will be one of our main concerns in our development of a local Goldie theory for Jordan algebras based on the ideas of \cite{fg1} and \cite{fg2}. We recall here that  if $J$ is nondegenerate, this set is an ideal of $J$ \cite[Theorem 5.13]{lc}. The next to sections will be devoted to the study of that ideal.
\end{apartado}

\begin{apartado}\label{strongly nonsingular ideal}We have already mentioned in \ref{aq} the possibility of developing a version of the theory of algebras of quotients based on essential inner ideals instead of on dense inner ideals following \cite{esenciales}. That requires the following version of nonsingularity introduced in \cite{esenciales}: a Jordan algebra $J$ is \emph{strongly nonsingular} if for any essential inner ideal $K$ of $J$, and any $a \in J$, the equality $U_aK=0$ implies $a=0$. Following \cite{tesis}, an element $z\in J$ will be called an \emph{essential zero divisor} if there exists an essential inner ideal $K$ of $J$ such that $U_zK=0$. Therefore, a Jordan algebra is strongly nonsingular if it does not have nonzero essential zero divisors. In that case, the theories developed in \cite{densos} and \cite{esenciales} coincide since by \cite[Lemma 1.18 (b)]{densos}, a Jordan algebra $J$ is strongly nonsingular if and only if any essential inner ideal of $J$ is dense.

Essential zero divisors can be gathered in an analogue of the singular ideal \ref{goldie stuff} whose vanishing will imply strong nonsingularity. Again following \cite{tesis}, we denote by ${\cal Z}_{ess}(J)$ the linear span of all essential zero divisors of $J$.
\end{apartado}

\begin{proposicion}\label{strongzeta}For any nondegenerate Jordan algebra $J$, the set
${\cal Z}_{ess}(J)$ is an ideal.
\end{proposicion}\begin{proof}We first prove that if $z$ is an essential zero divisor, then the inner ideal $(z) =\Phi z+ U_z\hat J$ generated by $z$ is contained in ${\cal Z}_{ess}(J)$. Indeed, if $U_zK =0$ for some essential inner ideal $K$, and $u = \alpha z+ U_za\in (z)$,  we have $U_uK \subseteq   \alpha^2U_zK+ \alpha
\{z,K,U_za\}+U_zU_aU_zK  = 0 + \{U_zK,a,z\}+ 0 =0$, hence $u \in {\cal Z}_{ess}(J)$.

Recall that a mapping $S:J\rightarrow J$ is a structural transformation if there exists $S^*:J\rightarrow J$ such that $U_{Sx} = SU_xS^\ast$ for any $x\in J$. For $x\in J$, the mappings $U_x$ are structural transformations, as are the mappings $B_x = Id-V_x+U_x$ (equal to $U_{1-x}$ in a unital $J$). Since $V_x = Id+B_x+U_x$, and the mapping $a\mapsto \{x,y,a\}$ is $V_xV_y-U_{x\circ y}+U_x+U_y$,
\cite[Lemma 4.4]{pi-ii} (or from \cite[Proposition 4.1.6]{jac-struc} applied to the unital $\hat J$) implies that the ideal generated by ${\cal Z}_{ess}(J)$ consists of the elements which are sums of elements of the form $W_1\cdots W_nz$, for $W_i= U_{a_i}$ or $B_{a_i}$, and an essential zero divisor $z$.

Now, if $z$ is an essential zero divisor with $U_zK=0$ for the essential inner ideal $K$, for any $a \in J$,

$$U_{U_az}(K:a) = U_aU_zU_a(K:a) \subseteq U_aU_zK =0$$
and
\begin{equation*}\begin{split}
U_{B_az}(K:a) &=B_aU_zB_a(K:a) \subseteq \\
&\subseteq
 B_a(U_z(K:a)+ U_z(a \circ  (K:a)) + U_zU_a(K:a))\\
&\subseteq
 U_aU_zK= 0.\end{split}\end{equation*}

Therefore both elements $U_az$ and $B_az$ are essential zero divisors, and thus any $W_1\cdots W_nz$ as above is an essential zero divisor, which proves that ${\cal Z}_{ess}(J)$ is an ideal of $J$.\end{proof}

\begin{lema}\label{strongzetalocal}Let $J$ be a nondegenerate Jordan algebra. For any $a\in J$, the following containment holds: $$({\cal Z}_{ess}(J)+Ker\,a)/Ker\,a \subseteq {\cal Z}_{ess}(J_a),$$
and therefore, ${\cal Z}_{ess}(J_a)=0$ implies $a\in ann_J({\cal Z}(J))$.\end{lema}
\begin{proof}Let $z\in J$ be an essential zero divisor, so that $U_zK =0$ for an essential inner ideal $K$ of $J$. Set $M = \{x\in J\mid U_ax\in K\}$, which is the preimage of $\overline{K} = (K+Ker\, a)/Ker\, a$ by the natural projection $J = J^{(a)}\rightarrow J_a$, $y \mapsto \bar{y}: = y+Ker\,a$. Clearly, since the set $M$ is the preimage of an inner ideal by that homomorphism, it is  an inner ideal of the homotope $J^{(a)}$.

Now, let $\overline{N} = N/Ker\,a$ be a nonzero ideal of $J_a$, where $Ker\,a\subseteq N$ is a nonzero ideal of $J^{(a)}$. Then $U_aN$ is a nonzero inner ideal of $J$ \cite{am-local}, hence by essentiality of $K$, there exists a nonzero $k= U_ax \in U_aN\cap K$, for some $0\ne x \in N$. Therefore $0\ne\bar  x \in \overline{N}\cap \overline{M}$, which proves that $\overline {M}$ is essential.

Finally, $U_{\bar z}\overline{M} = \overline{U_z^{(a)}M} = \overline{U_zU_aM} \subseteq \overline{U_zK} =0$, hence $\bar z \in {\cal Z}_{ess}(J_a)$, and thus $\overline{{\cal Z}_{ess}(J)} \subseteq {\cal Z}_{ess}(J_a) $, so we get the containment $\overline{{\cal Z}_{ess}(J)}\subseteq {\cal Z}_{ess}(J_a)$, and then $\overline{{\cal Z}_{ess}(J)}= 0$ implies $U_a{\cal Z}_{ess}(J) = 0$, hence $a \in ann_J({\cal Z}_{ess}(J))$.
\end{proof}

\begin{proposicion}\label{lckillsstrongzeta}If $J$ is a nondegenerate Jordan algebra, then $LC(J) \subseteq ann_J({\cal Z}_{ess}(J))$. Therefore, if $LC(J)$ is an essential ideal, then $J$ is strongly nonsingular.\end{proposicion}
\begin{proof}The second assertion follows directly from the first one since the containment $LC(J) \subseteq ann_J({\cal Z}_{ess}(J))$ implies ${\cal Z}_{ess}(J) \subseteq ann_J(LC(J))$. As for the first assertion, it suffices to prove that for any $a\in LC(J)$, $J_a$ is strongly nonsingular, according to \ref{strongzetalocal}, and this follows from the fact that if a Jordan algebra $J$ is LC, then it is strongly nonsingular. Indeed, if $J$ is LC, and $z\in J$ has $U_zK =0$ for an essential inner ideal $K$, then $K$ contains an injective element $s$, and $U_{U_sz}J = U_sU_zU_sJ \subseteq U_sU_zK=0$, hence $U_sz=0$ by nondegeneracy of $J$, and thus $z=0$ since $s$ is injective.\end{proof}

\begin{corolario}\label{densoesesencialenlc}Let $J$ be a nondegenerate Jordan algebra. If $J= LC(J)$ then an inner ideal of $J$ is essential if and only if it is dense.\end{corolario}
\begin{proof}This is straightforward from \ref{lckillsstrongzeta} and the coincidence of essentiality and density for inner ideals in strongly nonsingular algebras \ref{strongly nonsingular ideal}.\end{proof}

%
%

\section{Strongly prime Jordan algebras with nonzero LC-elements.}

In this section we consider nonzero LC-elements of strongly prime Jordan algebras and prove that   such elements  are exactly those elements of the Jordan algebra
lying in the socle of its maximal algebra of quotients.

\begin{apartado}  Let $J$  be a strongly prime Jordan algebra. If $J$ has nonzero PI-elements (i.e. nonzero elements at which the local algebra is a PI-algebra) then $LC(J)=PI(J)$ \cite[Proposition 3.3]{lc}. Otherwise, if $PI(J)=0$ (a situation that from now on, following \cite{fgm}, we will refer to by saying that $J$ is \emph{PI-less}), $J$ is special of hermitian type \cite[Lemma 5.1]{fgm} and given a $\ast$-tight associative envelope $R$ of $J$, the set of LC-elements coincides with the set of elements of $J$ having finite uniform dimension, more precisely, $LC(J)=F(J)=I_l(R)\cap J=I_r(R)\cap J$, where $F(J)=\{x\in J\mid udim(x)< \infty\}$,
$I_l(R)=\{x\in R\mid u_ldim(x) < \infty\}$, and
 $I_r(R)=\{x\in R\mid u_rdim(x) <\infty\}$ \cite[Theorem 4.4]{lc}. \end{apartado}

A consequence of the inheritance of density of inner ideals by their projections in local algebras proved in \cite{densos}, is that algebras of quotients interact well with taking local algebras:

\begin{lema}\label{local-q}
 Let $Q$ be a general algebra of quotients of a   Jordan algebra $J$.
For any $x\in J$, $Q_x$ is a general algebra of quotients of
$J_x$.\end{lema}

\begin{proof} Let $x\in J$. Then,  by \ref{aq}, $J$ is nondegenerate  and we have  $Ker_Jx=Ker_Qx\cap  J$. Therefore    $Q_x$ contains the
subalgebra $(J^{(x)}+Ker_Q x)/Ker_Qx$ isomorphic to $J_x$. We
denote with bars the projections in both $\overline{Q}=Q_x$ and
$\overline{J}=J_x$.

Take  $q\in Q$. It follows from  \ref{nii} and \cite[Lemma 1.20]{densos} that
$\overline{({\cal D}_J(q):x)}$ is a dense inner ideal of
$\overline{J} $. We will prove that ${\cal D}_{\overline{J} }(\overline{q})$ is dense in $\overline{J}$ by  checking the containment
  $\overline{({\cal
D}_J(q):x)}\subseteq {\cal D}_{\overline{J} }(\overline{q})$.

Take
$a\in ({\cal D}_J(q):x)$. By QJ16  we have
\begin{align*} U_aU_xq&=-U_xU_aq-V_aU_xV_aq+U_{a\circ
x}q+U_{U_ax,x}q\in\\%
&\in U_xU_{ {\cal D}_J(q) }q+V_aU_x({\cal D}_J(q)\circ q)+U_{{\cal
D}_J(q)}q+\{{\cal D}_J(q) ,q,J\}\subseteq J,\end{align*}%
hence
$U_{\overline{a}}\overline{q}=\overline{U_aU_xq}\in \overline{J}$.
On the other hand
\begin{align*} & U_{\overline{q}}\overline{a}=\overline{U_qU_xa}%
\in \overline{U_qU_x({\cal D}_J(q):x)}\subseteq \overline{U_q{\cal D}_J(q)}\subseteq \overline{J},\\
&U_{\overline{q}}U_{\overline{a}}\overline{J}=\overline{U_qU_xU_aU_xJ}%
=\overline{U_qU_{U_xa}J}\subseteq \overline{U_qU_{{\cal
D}_J(q)}J}\subseteq \overline{J},\end{align*} and
$$V_{\overline{q},\overline{a}}\overline{J}= \{\overline{q},\overline{a},\overline{J}\} %
=\overline{\{q,U_xa,J\}}\subseteq \overline{\{q,U_x({\cal
D}_J(q):x),J\}}\subseteq \overline{\{q,{\cal D}_J(q),J\}}\subseteq
\overline{J}$$ which implies  that  $\overline{({\cal
D}_J(q):x)}\subseteq {\cal D}_{\overline{J} }(\overline{q})$ by  \ref{denominadores}, and thus ${\cal
D}_{\overline{J}}(\overline{q})$ is a dense inner ideal of
$\overline{J}$.

Finally, if $U_{\overline{q}}{\cal
D}_{\overline{J}}(\overline{q})=\overline{0}$, then
$U_xU_qU_x({\cal D}_J(q):x)=0$ which implies that   $U_xq=0$ by the density of
${\cal D}_J(q)$ in $J$. Hence $\overline{q}=\overline{0}$ and therefore $\overline{Q}$ is a general algebra of quotients of $\overline{J}$.
\end{proof}

As noticed in the previous section classical algebras of quotients are general algebras of quotients. The converse holds for unital algebras with finite capacity.

 \begin{lema}\label{eq}  Let $J$ be   a subalgebra of a unital Jordan algebra
 $Q$ of finite capacity. Then the following assertions  are equivalent:
 \begin{enumerate}\item[(i)] $Q$ is an algebra of quotients of $J$.
 \item[(ii)] $J$ is a classical  order in $Q$.\end{enumerate}
 Moreover, under the above conditions both $J$ and $Q$ are nondegenerate.\end{lema}

 \begin{proof}
 Suppose that $Q$ is an algebra of quotients of $J$.  Since dense inner ideals are, in particular, essential \cite[Lemma 1.18 a]{densos},  $Q$ is an innerly tight extension
 of $J$  (see \ref{InTEx}) hence,  by \cite[Proposition 2.10]{fgm},  $J$ is a classical order in
 $Q$.
Conversely, if $J$ is a classical order in $Q$, then (i) follows
from \cite[Examples 2.3.5]{densos} and the fact that  $Q$ is generated by  $J$ and the
inverses in $Q$ of the elements of $Inj(J)$ since $J$ is a classical
order in $Q$.

The nondegenerancy of $J$ is a straightforward consequence of the above equivalent conditions, and that of $Q$ follows from
  \cite[Lemma 2.4(1)]{densos} if $Q$ is
an algebra of
 quotients of $J$, and from \cite[Proposition 2.9(vi)]{fgm}
 if $J$ is a classical order in $Q$.
  \end{proof}

 To prove the next proposition we
 need to recall a fact included in the
proof of \cite[Theorem 4.4]{lc}. Specifically, we recall the
following result  which first appeared in \cite[Theorem 6.5]{pi-i}.

\begin{lema}\label{tec} {\rm \cite[Theorem 6.5]{pi-i}} Let $J$ be a PI-less strongly prime Jordan algebra and $R$ a $\ast$-tight
associative envelope of $J$. Then, for each $a\in J$, the
subalgebra $S$ of $J$ generated by ${\cal H}(J)^{(1)}$, where
${\cal H}(J)\neq 0$ for some   hermitian ideal ${\cal H}(X)$ of
$FSJ(X)$, and the element $a$, is strongly prime of hermitian
type. Moreover, $A=alg_R(S)$ is a
$\ast$-tight associative envelope of $S$, and $S=H_0(A, \ast)$ is ample in $A$.\end{lema}

\begin{proposicion}\label{local-goldie}
Let $J$ be a PI-less strongly prime Jordan algebra and let $R$
a $\ast$-tight associative $\ast$-envelope of $J$.  If
$LC(J)\neq0$,  then $R_a$ is Goldie for any nonzero $a\in LC(J)$.
\end{proposicion}

 \begin{proof} Let $ a $ be a nonzero LC-element of $J$.  By \cite[Proposition 4.2]{lc}, $J$ is nonsingular and $J_a$ has finite uniform dimension, so that  by \cite[Theorem 4.4]{lc},
 $R_a$ has finite right and left uniform dimension, and thus for $R_a$ to be Goldie  it suffices that $R_a$ be right and left nonsingular.

 Since $J$ is PI-less, $J$ is special of hermitian type \cite[Lemma 5.1]{fgm}. Let ${\cal H}(X)$ be an hermitian ideal of
 $FSJ(X)$ such that  ${\cal H}(J)\neq0$ and then consider the subalgebra $S$ of $J$ generated by
 the ideal $I={\cal H}(J)^{(1)}$  and the element $a$.
 By \ref{tec} $S$ is a  strongly
 prime Jordan algebra and    $A=alg_R(S)$ is a $\ast$-tight associative envelope of
 $S$.

 Clearly, the local algebra $S_a$ contains a nonzero ideal  (hence an essential ideal)  which is isomorphic to the nonzero ideal
 $(I+Ker_Ja)/Ker_Ja$ of the strongly prime Jordan algebra $J_a$. Since $LC(J)\neq0$, $J$ is nonsingular
 \cite[Proposition 4.2]{lc}, and so is the local algebra $J_a$
 \cite[Lemma 4.1]{lc}. Therefore $S_a$ is nonsingular by \cite[Proposition 6.2]{fgm}. Moreover we have $udim(J_a)=udim(S_a)$, which implies that  $S_a$ has finite uniform dimension.

Next we claim that  $alg_{A_a}(S_a)$, the $\ast$-subalgebra of $A_a$ generated by $S_a$, is $\ast$-tight over $S_a$  and that  $alg_{A_a}(S_a)$ contains a nonzero $\ast$-ideal of $A_a$. This claim is proved in \cite[Theorem 4.4]{lc}. Now since $alg_{A_a}(S_a)$ is $\ast$-tight over $S_a$ and $S_a$ has finite uniform dimension, the nonsingularity of  $alg_{A_a}(S_a)$ follows from that of $S_a$ by \cite[Theorem 7.17]{fgm}.
Finally since $alg_{A_a}(S_a)\subseteq A_a\subseteq R_a\subseteq
Q_s(A_a)$, where $Q_s(A_a)$ denotes the algebra of symmetric Martindale ring of quotients of $A_a$    and  $alg_{A_a}(S_a)$ contains a nonzero $\ast$-ideal of
$A_a$,  we have that $R_a$ is nonsingular, hence $R_a$ is Goldie.
 \end{proof}

   \begin{lema}\label{L(as)} Let $R$ be a semiprime associative algebra. Let $a\in R$ be such that $R_a$ is Goldie. Then for any $s\in R$ such that $\overline{s}\in Reg(R_a)$:
  \begin{enumerate}
  \item[(i)] $L(as)=Ras+l_R(a)$, where $l_R(a)$ denotes the left annihilator  of
$a$ in $R$, is a dense left ideal of $R$.
    \item[(ii)] The pair $ (L(as), f_s)$, where the map $f_s: L(as)\to R$ is given by $f_s(xas+z)=xa$, for all $x\in R$ and $z\in l_R(a)$, defines an element in $Q_{max}^l(R)$, the maximal left algebra of quotients of $R$.
  \end{enumerate}
  \end{lema}

  \begin{proof}  Let $a\in R$ be such that $R_a$ is Goldie and take $s\in R$ whose projection $\overline{s}$ in $R_a$ is regular in $R_a$.

  (i) $L(as)$ is clearly a left ideal of $R$. To prove its density, take $b,c\in R$ and assume that $(L(as):b)c=0$. Take
  $\overline{x}\in (R_a\overline{s}:\overline{b})$. Then we have
  $\overline{xab}=\overline{x}\overline{b}\in R_a \overline{s}=\overline{Ras}$, thus there is
    $y\in R$ such that $axaba=ayasa$, hence $axab=ayas-z$ for some   $z\in l_R(a)$ and $axab\in Ras+l_R(a)=L(as)$.
   Therefore, for all $\overline{x}\in (R_a \overline{s}:\overline{b})$ we have $axa\in (L(as):b)$ and thus
   $(L(as):b)c=0$ implies that $(R_a\overline{s}:\overline{b})\overline{c}=\overline{0}$. But  since
   $R_a$ is Goldie,  $R_a\overline{s}$ is dense (by the regularity of $ \overline{s}$ in $R_a $) and we get
    $\overline{c}=\overline{0}$, or equivalently,
  $aca=0$.

   Now, for any  $r, t\in R$. we have $(L(as):rb)rct\subseteq
 (L(as):b)ct=0$, and as above, we get $arcta=0$. Since
     $r, t\in R$ are arbitrary, we get $aRcRa=0$, hence $(RcRa)R(RcRa)=0$, and from the
 semiprimeness of $R$, it   follows that the ideal $RcR$ is
 contained in $l_R(a)$. Then $RcR\subseteq (L(as):b)$, and we have
  $RcRc=0$ since $(L(as):b)c=0$, which again by  semiprimeness  of $R$ implies $c=0$. Thus, $L(as)$ is a dense left ideal of $R$.

    (ii) We claim that the pair $ (L(as), f_s)$ defines an element of $Q_{max}^l(R)$, so we begin by proving that $f_s$ is a well-defined homomorphism of left $R$-modules. To that end, we first note
    that  $Ras\cap l_R(a)=0$. Indeed, if
 $xas\in l_R(a)$ with $x\in R$, for all $r\in R$ we have $rxasa=0$,
 hence $arxasa=0$, and thus $\overline{(rx)}\overline{s}=\overline{0}$ in $R_a$. But since $\overline{s}$ is regular, then $\overline{rx}=\overline{0}$. Therefore
   $aRxa=0$, and $xa=0$ by
 semiprimeness of $R$, hence  $xas=0$.

 Thus, to prove that $f_s$ is well-defined it suffices to check that if $xas=0$ for some $x\in R$, then $xa=0$. But this can be proved as before.
Thereby, $f_s$ is   well-defined,  and since it is clearly a homomorphism of left $R$-modules, from the density of $L(as)$ in $R$ it follows that
 $ (L(as), f_s)$ defines an element of  $Q_{max}^l(R)$.
    \end{proof}

\begin{apartado}\label{elementoqs}
Let us denote by $q_s$ the element of $Q_{max}^l(R)$  determined by the pair $( L(as), f_s)$ described above.  It is straightforward to check that $q_s$ does indeed belong to the maximal symmetric algebra of quotients $Q_\sigma(R)$ of $R$.
\end{apartado}

\begin{lema}\label{qsQsigma} Let $R$ be a semiprime associative algebra and  $Q_\sigma(R)$ be its maximal symmetric algebra of quotients. For any $a \in R$ such that $R_a$ is Goldie, and any $s\in R$ with $\overline{s}\in Reg(R_a)$, the pair $( L(as), f_s)$  defines an element of $Q_\sigma(R)$.  \end{lema}

\begin{proof} Let $K(sa)=saR+r_R(a)$, where $r_R(a)$ is the right annihilator of $a$ in $R$. Following \ref{L(as)} we can prove that $K(sa)$ is a dense right ideal of $R$. We claim that
$q_sK(sa)\subseteq R$.

Take $sax+u\in K(sa)$ with $x\in R$ and $u\in r_R(a)$. Then, for any
$y\in R$ and $z\in l_R(a)$, we have
$(yas+z)q_s(sax+u)=ya(sax+u)=yasax=(yas+z)ax$, hence
$L(as)(q_s(sax+u)-ax)=0$ which by the density of $L(as)$ implies that
$q_s(sax+u)=ax$. Hence $q_sK(sa)\subseteq R$, and it follows from
\cite{lanning} that $q_s$ does indeed belong to $Q_\sigma(R)$.
\end{proof}

 \begin{proposicion}\label{goldie-socle} Let $R$ be a semiprime associative algebra  with maximal symmetric algebra of quotients $Q_\sigma(R)$. For any $a \in R$ such that $R_a$ is Goldie  we have
 $a\in Soc( Q_\sigma(R))$.
 \end{proposicion}

\begin{proof}
Take $a\in R$ and assume that $R_a$ is Goldie.
Take $s\in R$ with $\overline{s}\in Reg (R_a)$ and consider  $q_s\in Q_\sigma(R)$ as defined in \ref{elementoqs}. Clearly
$asq_s=q_ssa=a$ and $l_R(a)
q_s=q_sr_R(a)=0$. Since $\overline{sas}=\overline{s}\ \overline{s}\in Reg(R_a)$, the element $q_{sas}$ is also defined. We claim that $u=sq_{sas}s$ satisfies $aua=a$. Indeed
since $asasq_{sas}=a$, for all $x\in R$ and $z\in l_R(a)$ we have $(xas+z)(asq_{sas}sa-a)=xasa-xasa=0$, hence $L(as)(aua-a)=0$ which implies  $aua=a$ by the density of $L(as)$ in $R$.
As a result we get that $Q_\sigma(R)_a$ is unital with unit $\overline{u}$.

Now since  $Q_\sigma(R)_a$ is an algebra of quotients of $R_a$, we have  $Q_\sigma(R)_a\subseteq Q_{max}(R_a)$ (here $Q_{max}(R_a)$ stands for any the left or right maximal algebra of quotients of $R_a$). But $R_a$ being Goldie,
$ Q_{max}(R_a)$ is the algebra of $Reg(R_a)$-fractions of $R_a$ , hence it is generated by $R_a$ and the inverses of the elements of $Reg(R_a)$. Thus, to ensure that
$Q_\sigma(R)_a= Q_{max}(R_a)$ it is enough to prove that every element of  $Reg(R_a)$ is invertible in $Q_\sigma(R)_a$. Take  $\overline{s}\in Reg (R_a)$ and set $p=sq_{sas}u$ with $u=sq_{sas}s$.
 Then
$asapa=asasq_{sas}ua=aua$ and so $\overline{s}\ \overline{p}=\overline{u}$ in  $Q_\sigma(R)_a$. Hence $\overline{p}=\overline{s}^{-1}$.
Therefore $Q_\sigma(R)_a= Q_{max}(R_a)$, and since $R_a$ is Goldie, $Q_\sigma(R)_a= Q_{max}(R_a)$ is artinian (and semiprime) which implies that
$a\in Soc(Q_\sigma(R))$.
\end{proof}

 \begin{teorema}\label{casofprimo}
Let $J$ be a strongly prime Jordan algebra. Then $LC(J)=Soc
(Q_{max}(J))\cap J$.
 \end{teorema}
  \begin{proof}
If $J$ has nonzero PI-elements this follows directly from  \cite[Theorem 5.1]{pi-ii} since by \cite[Proposition 3.3]{lc} we have  $LC(J)=PI(J)$,  and    by \cite[Lemma 3.1(c), Theorem 3.5]{esenciales}   we have
$Q_{max}(J)=\Gamma^{-1}(J)J$. Therefore we can
assume that $J$ has no nonzero PI-elements. As a consequence, $J$ is special
\cite[Lemma 5.1]{fgm}, and we can fix a $\ast$-tight associative envelope $R$ of  $J$.

Let $a\in  LC(J)$. By \ref{local-goldie}, $R_a$ is Goldie, and therefore, by \ref{goldie-socle},
  $a\in Soc(Q_\sigma(R))$. Now since by
\cite[Theorem 4.10]{densos}, $Q_{max}(J)=H_0(Q_\sigma(R),\ast)$,  by
\cite[Proposition 4.1(2)]{ft}  we have
 $ a\in  Soc(Q_\sigma(R))\cap
J=Soc(Q_\sigma(R))\cap Q_{max}(J)\cap J=  Soc(Q_\sigma(R))\cap
H_0(Q_\sigma(R),\ast)\cap J =Soc(H_0(Q_\sigma(R),\ast))\cap
J=Soc(Q_{max}(J))\cap J$.

Conversely take $a\in Soc(Q_{max}(J))\cap J$. Then  by \ref{local-q} $Q_{max}(J)_a$ is
an algebra of quotients of $J_a$, which has finite capacity by \cite[Lemma 0.7(b)]{pi-i}. Thus $J_a$ is a classical order
in $Q_{max}(J)_a$ by \ref{eq},
  which implies that   $a\in LC(J)$.
  \end{proof}

%
%
\section{Nondegenerate Jordan algebras with nonzero LC-elements.}

\begin{apartado} Recall that a \emph{subdirect
product} of a collection of Jordan algebras $\{J_\alpha\}$ is any
subalgebra $J$ of the full direct product $\prod_\alpha J_\alpha$ of the algebras of that collection such that
the restrictions of the canonical projections $\pi_\alpha: J\to J_\alpha$ are onto. An
\emph{essential subdirect product} is a subdirect product which
contains an essential ideal of the full direct product. If $J$ is
actually contained in the direct sum of the $J_\alpha$, then $J$
is called an \emph{essential subdirect sum} \cite[p.448]{fgm}.
\end{apartado}

\begin{apartado}  Local algebras at nonzero LC-elements of   nondegenerate Jordan
algebras are essential subdirect products of finitely many
strongly prime Jordan algebras, and therefore they are essential direct sums of the corresponding factors. That is the case for the local algebra $J_a$ at a nonzero
LC-element $a$ of a nondegenerate Jordan algebra $J$. Indeed, by  \cite[Theorem 5.13]{lc}, \cite[10.3]{fgm}, and
\cite[Lemma 5.3]{lc}, the local algebra $J_a\cong
\Big(J/ann_J(id_J(a))\Big)_{a+ann_J(id_J(a))}$ is an essential
subdirect sum of finitely many strongly prime Jordan algebras.
These essential subdirect sums arise from the fact that for any LC-element $a\in  LC(J)$ the ideal $id_J(a)$ is  semi-uniform \cite[Proposition  5.10]{lc} in the following sense:  an ideal $I$ of a nondegenerate Jordan algebra $J$ is \emph{semi-uniform} if there exist prime ideals $P_1,\ldots,P_n$ of $J$  such that $P_1\cap \ldots \cap P_n\subseteq ann_J(I)$. For any semi-uniform ideal $I$ of a nondegenerate Jordan algebra $J$ there exists a unique minimal set of prime  ideals ${\cal P}=\{P_1,\ldots,P_n\}$ with $P_1\cap \ldots \cap P_n= ann_J(I)$, named (after the similar situation appearing in commutative ring theory) the \emph{set of prime ideals associated to} $I$.
\end{apartado}

\begin{lema}\label{quot_essential_ideal}
Let $J$ be a nondegenerate Jordan algebra. For any essential ideal $I$ of $J$ we have $Q_{max}(I)=Q_{max}(J)$.
\end{lema}

\begin{proof}
Let $I$ be an essential ideal of $J$.  Then, by \cite[Examples 2.3.1]{densos}
 $J$ is an algebra of quotients of $I$
 and therefore $I\subseteq J\subseteq Q_{max}(J)$ is a sequence of algebras of quotients which implies by   \cite[Proposition 2.8]{densos} that
 $Q_{max}(J)$ is an algebra of quotients of $I$.
 But since $I$ is nondegenerate \cite{mc-inh}, $I$ has a maximal algebra of quotients $Q_{max}(I)$ whose  maximality implies that $Q_{max}(J)\subseteq Q_{max}(I)$ \cite[Lemma 2.12]{densos}.

  Now, again by \cite[Proposition 2.8]{densos} we have that $Q_{max}(I)$ is an algebra of quotients of $J$, and thus by the maximality of $Q_{max}(J)$ \cite[Lemma 2.12]{densos}, it follows that  $Q_{max}(I)\subseteq Q_{max}(J)$. Hence
$Q_{max}(I)=Q_{max}(J)$
\end{proof}

\begin{lema}\label{qideales} Let J be a nondegenerate Jordan algebra, and let $I$ and $L$ be ideals of $J$ such that $J=I\oplus L$. Then $Q_{max}(J)=Q_{max}(I)\oplus Q_{max}(L)$.
\end{lema}

\begin{proof}
The nondegeneracy of $J$ implies that $I=ann_J(L)$ and $L=ann_J(I)$, and since $L$ is isomorphic to an essential ideal of the nondegenerate Jordan algebra $J/ann_J(I)$ \cite[Lemma 1(i)]{fg2}, by \ref{quot_essential_ideal} we have $Q_{max}(L)=Q_{max}(J/ann_J(I))$, and the result  follows from \cite[Lemma 5.6]{densos}.
\end{proof}

We next extend \ref{casofprimo} to nondegenerate algebras.

\begin{teorema}\label{caso general}
Let $J$ be a nondegenerate Jordan algebra and $Q_{max}(J)$ be its maximal algebra of quotients.    Then
$LC(J)=J\cap Soc(Q_{max}(J))$.
\end{teorema}

\begin{proof}
Assume first that $LC(J)=0$ and take $a\in J\cap Soc(Q_{max}(J))$. By \ref{local-q}, $Q_{max}(J)_a$ is an algebra of quotients of $J_a$, and since $a\in Soc(Q_{max}(J))$,
$Q_{max}(J)_a$ has finite capacity. Thus  $a\in LC(J)=0$  by \ref{lc def}.

Assume now that  $LC(J)\neq 0$, and
let  $a$ be nonzero LC-element of $J$.
By \cite[10.3]{fgm} and \cite[Proposition 5.6]{lc}, $J_a$ is semi-uniform, and so is the ideal $id_J(a)$ generated by $a$ in $J$ by \cite[Proposition 5.10 (ii)]{lc}. Hence
$\overline{J}=J/ann_J(id_J(a))$ is an
essential subdirect sum of finitely many strongly prime Jordan
algebras $J_1,\ldots, J_n$  \cite[Proposition 5.4]{lc}, that is, $M\subseteq _{ess}
\overline{J}\cong J_1\oplus\cdots\oplus J_n$,  where $M$ is an
essential ideal of $J_1\oplus\cdots\oplus J_n$ contained into
$\overline{J}$ (see the proof of \cite[Theorem 5.13]{lc}).
 By \ref{quot_essential_ideal}, we have $ Q_{max}(M)\cong Q_{max}(\overline{J})\cong
Q_{max}(J_1\oplus\cdots\oplus J_n)$. Moreover   \cite[Remark 5.7]{densos} applies here implying
that $Q_{max}(\overline{J})\cong Q_{ max}(J_1)\oplus\cdots\oplus
Q_{max}(J_n)$ since $J_i=J/P_i$,
where $\{P_1,\ldots,P_n\}$ is the set of minimal prime ideals
associated to $id_J(a)$ (see also \ref{qideales}).

Write $\overline{a}=\overline{a}_1+\cdots+\overline{a}_n$ with
$\overline{a}_i\in J_i$. By \cite[Proposition 5.12]{lc}, we have $\overline{a}_i\in
LC(J_i)$,   and since $J_i$ is strongly prime,  $\overline{a}_i\in Soc(Q_{max}(J_i))\cap J_i$ by
\ref{casofprimo}. Therefore we get
$\overline{a}=\overline{a}_1+\cdots+\overline{a}_n\in
Soc(Q_{max}(J_1))\oplus\cdots\oplus
Soc(Q_{max}(J_n))=Soc(Q_{max}(J_1)\oplus\cdots\oplus
Q_{max}(J_n))=Soc(Q_{max}(\overline{J}))$.

On the other hand, the Jordan algebra $\overline{J}$ contains the essential ideal
$(id_J(a)+ann_J(id_J(a)))/ann_J(id_J(a))$, isomorphic to $id_J(a)$. Thus we have $Q_{max}(\overline{J})\cong
Q_{max}(id_J(a))$ by \ref{quot_essential_ideal}, and therefore, $a\in Soc(Q_{max}(id_J(a)))\cap
id_J(a)$. Finally, making use again of  \ref{quot_essential_ideal} gives $a\in
Soc(Q_{max}(J))\cap J$.

As for the reverse containment $Soc(Q_{max}(J))\cap J\subseteq LC(J)$, it suffices to note that the proof of \ref{casofprimo} still works here.
\end{proof}

%
%

\section{Local orders in Jordan algebras}

 In order to do this section as self-contained as possible we include here the quadratic versions of some of the results proved in \cite{fg1,fg2} for
 linear Jordan algebras, but not always their complete proofs
 as they can be easily obtained  from those given in
 \cite{fg1}. However we outline some of those proofs in order to stress the necessary quadratic references.

\begin{apartado}\label{localyinvert}  An element $x$ of a Jordan algebra $J$ is said to be
\emph{locally invertible} if there exists a (necessarily unique) idempotent
$e=P(x)$ such that $x$ is invertible in the unital Jordan algebra
$U_eJ$. We denote by  $  LocInv(J)$ the set of all locally
invertible elements of $J$. The inverse $x^\sharp$ of $x$ in
$U_eJ$ is called the \emph{generalized inverse} of $x$.
The following equivalent characterizations of $x^\sharp$ are given in
\cite{fgss} (see also \cite[p. 1033]{fg1}):
\begin{enumerate}
\item[(i)] $U_xx^\sharp=x$, $U_{x^\sharp}x=x^\sharp$ and $U_{x^\sharp}U_x=U_xU_{x^\sharp}$.
\item[(ii)] $U_xx^\sharp=x$ and $U_xU_{x^\sharp}x^\sharp=x^\sharp$.
\item[(iii)] $U_xx^\sharp=x$ and $(x^\sharp)^2\circ x=x^\sharp$.
\end{enumerate}
 The idempotent $e=P(x)$ determined by the locally invertible element $x\in J$ is given by
$e=U_x(x^\sharp)^2=U_{x^\sharp}x^2$.  Recall that $x\in J$ is locally invertible if and only if $x$ is strongly regular (i.e. $x\in U_{x^2}J$) \cite[p. 1032]{fg1}.
\end{apartado}

\begin{apartado} A subalgebra $J$ of a (non necessarily unital)
 Jordan algebra $Q$ is a \emph{weak local order} in $Q$
if for each $q\in Q$ there exists $x\in LocInv(Q)\cap J$ such that
$q\in U_xQ$ with $U_xJ$ being an order in the unital Jordan algebra
$U_xQ=U_eQ$ ($e=P(x)$) relative to some monad $S$ of $U_xJ$.
\end{apartado}

The notion of weak local order was introduced for non-necessarily unital Jordan algebras as substitute of the notion of  order relative to a monad in nonunital Jordan algebras. These are indeed particular cases of weak local orders.

\begin{proposicion}\label{proposicion 12 de fg1} Let $J$ be a Jordan algebra which is
an $S$-order in a unital Jordan algebra $Q$. For every $s\in S$, $U_sS$ is a monad of $U_sJ$, and
$U_sJ$ is an $U_sS$-order in $Q$. In particular $J$ is a
weak local order in $Q$.
\end{proposicion}
\begin{proof} A proof similar to that of \cite[Proposition 12]{fg1} works here, replacing the
 linear references used there by the their quadratic counterparts
which can be found for example in \cite[Lemma 2.2]{fgm}.
\end{proof}

\begin{apartado} The concept of local order upon which the theory developed in \cite{fg1, fg2} lies is based on the following notion adapted from the associative theory. An element $x\neq 0$ in a Jordan algebra $J$ is called \emph{semiregular} if $ann_J(x)=ann_J(x^2)$. The set of all semiregular elements of $J$ is denoted by $SemiReg(J)$. In the present paper, however, we will adopt a different Jordan analogue of the homonymous arrogative notion that follows the approach of \cite{fgm} to Jordan rings of fractions, and in particular, the use of the notion of injective element instead that of regular element.  An element $x\in J$ will be said to be \emph{ semi-injective} if for any $y\in J$, $U_x^2y=0$ implies $U_xy=0$, which implies (and is equivalent, if $J$ is nondegenerate) the condition $Ker\, x=Ker\,x^2$. We denote by $SemiInj(J)$ the set of all nonzero semi-injective elements of $J$.

For a subalgebra $J$ of a Jordan algebra $Q$, the containments $LocInv(J)\subseteq
J\cap LocInv(Q)\subseteq SemiInj(J)\subseteq SemiReg(J)$  are straightforward. Moreover the equality holds for nondegenerate Jordan algebras satisfying the descending chain condition on principal inner ideals, since  \cite[Proposition 17]{fg1} remains true in the quadratic setting. The proof is essentially the one given there, taking into account the obvious changes required by the differences in the notions of annihilators in quadratic and linear Jordan algebras (see \ref{annihilator_def}).
\end{apartado}

Semi-injectivity has the following straihgtforward consequence in terms of local algebras:

\begin{lema}\label{localsemiinject}Let $J$ be a Jordan algebra. If $x\in SemiInj(J)$ then the subalgebra $U_xJ$ of $J$ is isomorphic to the local algebra $J_{x^2}$.
\end{lema}
\begin{proof}Since the mapping $U_x:J\rightarrow U_xJ$ obviously defines a surjective homomorphism $J^{(x^2)}\rightarrow U_xJ$, it suffices to notice that its kernel is $Ker\,x= Ker\,x^2$ since $x$ is semi-injective.
\end{proof}

\begin{apartado}  A subalgebra $J$ of a Jordan algebra $Q$ will be said to be a \emph{local order}
in $Q$ if it satisfies:
\begin{enumerate}\item[(LO1)]$SemiInj(J)\subseteq LocInv(Q)$, and
\item[(LO2)]for every $q\in Q$
there exists $x\in SemiInj(J)$ such that $q\in U_xQ$, and $U_xJ$
is a classical order in $U_eQ=U_xQ$ ($e=P(x)$).\end{enumerate}\end{apartado}

Property (LO2) in the definition of local order includes the assertion that, for some $x\in SemiInj(J)$, $U_xJ$ is a classical order in $U_xQ$. This can be replaced by the following assertion:

\begin{lema}\label{quotlocalsemiinj}Let $J$ be a subalgebra of an algebra $Q$. If $x\in SemiInj(J)$, $U_xJ$ is a classical order in $U_xQ$ if and only if $J_{x^2}$ is a classical order in $Q_{x^2}$ (with the obvious identification $(J_{x^2} = J+Ker_Qx^2)/Ker_Qx^2$.)
\end{lema}
\begin{proof}Considering the homomorphism $Q^{(x^2)}\rightarrow U_xQ$ given by $q\mapsto U_xq$ and its restriction to $J$ instead of $Q$, the assertion follows directly from \ref{localsemiinject} and the fact that $Ker_Qx\cap J=Ker_Qx^2\cap J = Ker_Jx= Ker_Jx^2$.
\end{proof}

\begin{lema}\label{semiinjgotoquotients}Let $J$ be a nondegenerate Jordan algebra. If $Q\supseteq J$ is a general algebra of quotients of $J$ then $SemiInj(J)\subseteq  SemiInj(Q)$\end{lema}
\begin{proof}Take $x\in SemiInj(J)$, and suppose that $U_{x^2}q =0$ for some $q\in Q$. Set $K = ({\cal D}_J(q):x)\cap ({\cal D}_J(q):x^2)$. Since $Q$ is an algebra of quotients of $J$, ${\cal D}_J(q)$ is a dense inner ideal, hence both $({\cal D}_J(q):x)$ and  $({\cal D}_J(q):x^2)$ are dense, and thus, so is their intersection $K$ (see \cite[Lemma 1.10]{densos}). Now we have $0 =U_{U_x^2q}K = U_{x^2}U_qU_{x^2}K$, but $U_qU_{x^2}K\subseteq U_qU_{x^2}({\cal D}_J(q):x^2)\subseteq  U_q{\cal D}_J(q) \subseteq J$, which implies that $U_xU_qU_{x^2}K= 0$ because $x \in SemiInj(J)$. Thus, for an arbitrary $z\in K$ we have $U_{U_{x^2}U_qU_xz}K = U_{x^2}U_qU_xU_zU_xU_qU_{x^2}K = 0$ by the previous equality, hence $U_{x^2}U_qU_xz=0$ by \cite[Lemma 2.4(iv)]{densos} for all $z \in K$, that is $U_{x^2}U_qU_xK=0$; but $U_{x^2}U_qU_xK \subseteq U_{x^2}U_qU_x({\cal D}_J(q):x)$, and since $U_qU_x({\cal D}_J(q):x) \subseteq  J$, the semi-injectivity of $x$ in $J$ implies $U_xU_qU_xK=0$. Then  $U_{U_xq}K=0$, which again by \cite[Lemma 2.4(iv)]{densos}, implies $U_xq =0$, and therefore $x \in SemiInj(Q)$. \end{proof}

Local  orders were  defined for non-necessarily
unital Jordan algebras, but when the involved  Jordan algebras are unital, local
orders are  algebras of quotients in the sense of
\cite{densos} and \cite{esenciales}, which coincide in this case. This will be obtained as a consequence of the following:

\begin{lema}\label{intersection inner ideals}
Let $J$ be a nondegenerate Jordan algebra. If $J$ is a subalgebra
of a Jordan algebra $Q$ and one of the following situations holds:
\begin{enumerate}\item[(i)] $J$ is a (weak) local order in $Q$,  \item[(ii)] $Q$ is an algebra
of quotients of  $J$,\end{enumerate} then $U_qJ\cap J\neq0$ for
any $0\neq q\in Q$, and as consequence, $Q$ is nondegenerate, and any nonzero inner ideal of $Q$ hits $J$ nontrivially.
\end{lema}
\begin{proof} (i) Since $J\subseteq Q$ is a (weak)
local order (that is, $J\subseteq Q$ satisfies (LO2)), given $q\in Q$, there exists an element $x\in LocInv(Q)\cap J \subseteq SemiInj(J)$ such that $q\in U_xQ$, and $U_xJ$
is a classical order in $U_eQ=U_xQ$ ($e=P(x)$). Now, since $J$ is nondegenerate, \ref{quotlocalsemiinj} together with \cite[Proposition 0.2]{acm} implies that $U_xJ$ is nondegenerate, and thus \cite[Proposition 2.9 (iii)]{fgm} implies that $U_qU_xJ\cap U_xJ 	\neq 0$, which obviously implies $U_qJ\cap J \neq 0$.

 The assertion in case (ii) is just \cite[Lemma 2.4(ii)]{densos}, and the last assertions are obvious.\end{proof}

\begin{proposicion}\label{localorderqotient} Let $J$ be a nondegenerate Jordan algebra and a local order in a Jordan algebra $Q$. Then:
\begin{enumerate}\item[(i)] if $J$ is unital,  $Q$ is also unital with the same unit as $J$.
\item[(ii)]
if  $Q$ is unital,  $J$ is
a classical order in $Q$ and therefore $Q$ is an algebra of quotients of
$J$.\end{enumerate}\end{proposicion}

\begin{proof} The proof of \cite[Lemma 3.2.(a)]{esenciales} applies here using \ref{intersection inner ideals}(i), the same property as \cite[Lemma 2.4 (iii)]{esenciales}, which is used in \cite[Lemma 3.2.(a)]{esenciales}). For (ii), the proof of
\cite[Proposition 19]{fg1} also works here with the obvious changes for the references. The last statement follows from
\cite[Examples 2.3.5]{densos}.
\end{proof}

We next aim at giving an alternative characterization of local orders in nondegenerate algebras with dcc on inner ideals which suggests a reason of the use of the adjective "local". We first prove a result which, among other uses, will be instrumental to that end.

\begin{lema}\label{localofclassicalordersatsocle} Let the Jordan algebra $J$ be a classical order in a nondegenerate Jordan algebra $\tilde J$. If $a\in Soc(\tilde J)\cap J$, then under the natural identification $J_a\subseteq \tilde J_a$ induced by the inclusion $J\subseteq \tilde J$, $\tilde J_a$ is a classical algebra of quotients of $J_a$.
\end{lema}
\begin{proof}According to \ref{localorderqotient} $\tilde J$ is an algebra of quotients of $J$, hence $J_a$ is an algebra of quotients of $\tilde J_a$. Since $a\in Soc(\tilde J)$, $\tilde J_a$ has finite capacity by \cite[Lemma 0.7(b)]{pi-i}, therefore $J_a$ is a classical order
in $\tilde J_a$ by \ref{eq}.
\end{proof}

\begin{apartado}\label{almostpeirce}We introduce a piece of notation that will be used in the sequel. Let $J$ be a subalgebra of a Jordan algebra $\tilde J$, and let $e\in \tilde J$ be an idempotent. We denote $J_2(e) = \{a\in J\mid U_ea=a\} = \tilde J_2(e)\cap J$ which is clearly an  inner ideal of $J$ because $\tilde J_2(e)$ is an inner ideal of $\tilde J$.\end{apartado}

The following extended version of Litoff's Theorem for quadratic Jordan algebras, whose linear version was proved in \cite{anh}, will be a key tool in our study of quadratic Jordan algebras which are local orders in nondegenerate Jordan algebras satisfying dcc on principal inner ideals.

\begin{teorema}\label{litoff_QJA}
Let $J$ be a nondegenerate Jordan algebra with $J=Soc(J)$. Then for every finite collection of elements $a_1,\ldots,a_n$ in $J$ there exists an idempotent $e\in J$ such that $a_1,\ldots,a_n\in U_eJ=J_2(e)$.
\end{teorema}
\begin{proof} Since $J$ equals its socle,  \cite[Theorem 2(b)]{loos-socle}, $J$ is a direct sum of simple ideals (see \ref{socle-def}), and we can therefore assume that $J$ is simple. Now the result follows directly if $J$ has finite capacity (since then $J$ is unital), so by the Simple Structure Theorem \cite[Theorem 15.5]{mcz},
 we can assume that either $J=A^+$ of $J=H_0(A, \ast)$ for a simple associative algebra $A$. Moreover by \cite[Proposition 4.1]{ft} we have $A=Soc(A)$. Now the proof of \cite{anh}
 carries over unchanged to the quadratic setting since the only elements of $H(A,\ast)$ whose presence in the subalgebra of symmetric elements is assumed in \cite{anh} are either norms or traces, and therefore they belong to any ample $H_0(A,\ast)$. \end{proof}

\begin{lema}\label{InTEx_idempotents}  Let $J$ be a nondegenerate Jordan algebra, and let $\widetilde{J}\supseteq J$ be an innerly tight extension of $J$ with $J=Soc(J)$. Then for any idempotent $e\in Soc(\widetilde{J})$ there exists $x\in SemiInj(J)$ such that $e=P(x)$ (i.e. $x$ is invertible in $U_e\widetilde{J}$).
\end{lema}
\begin{proof}We consider the double pair $V(\tilde J) = (\tilde J,\tilde J)$. Since $\tilde J$ equals its socle, so does $V(\tilde J)$. According to \cite{loos-finiteness} we can take a strong frame $\big\{e_1=(e_1^+,e_1^-),\ldots,e_n=(e_n^+,e_n^-)\big\}$ of the nondegenerate Jordan pair $\big(\widetilde{J}_2(e),\widetilde{J}_2(e)\big)$, where $n$ is the capacity of the Jordan pair \cite[Lemma 0.7(b)]{pi-i}.

The inner tightness of $\widetilde{J}$ over $J$ implies  (see \ref{InTEx}) that there exist nonzero elements $0\neq k_i\in U_{e_i^+}\tilde J\cap J$ with $k_i=U_{e_i^+}U_{e_i^-}k_i$. Let $x=k_1+\cdots +k_n\in J$. Then $U_ex=x$, so that $x\in U_e\widetilde{J}$ and we have  $rk(x)=rk(e)=n$, so that
$x$ is invertible in $U_e\widetilde{J}$ \cite[Proposition 1, Corollary 1]{loos-diagonal}. Hence   $x\in J\cap LocInv(\widetilde{J})\subseteq SemiInj(J)$.
\end{proof}

The following result contains what will be our operating version of local orders since we will be mainly interested in local orders of algebras with dcc on principal inner ideals, that is on algebras that equal their socles. As mentioned in the introduction, and in analogy with \cite[Theorem 1]{anhm1},  later on we will prove that these algebras are in fact general algebras of quotients.

\begin{teorema}\label{whylocalarelocal}Let $J$ be nondegenerate Jordan algebra, and let $Q\supseteq J$ be an over-algebra such that $Q = Soc(Q)$. Then $J$ is a local order in $Q$ if and only if:
\begin{enumerate}\item[(LOS1)]$SemiInj(J) = LocInv(Q)\cap J$,
\item[(LOS2)]For any $q\in Q$ there exists $x\in SemiInj(J)$ such that $q\in U_xQ$, and
\item[(LOS3)]For any $x\in J$, the local algebra $J_x$ is a classical order in $Q_x$ (with the obvious containment).
\end{enumerate}\end{teorema}
\begin{proof}Suppose first that $J$ is a local order in $Q$. As LOS1 and LOS2 are part of the definition of local order, it is clear that it suffices to prove that  LOS3 holds. Take  $x\in J$. By the definition of local order, there exists an element $s\in SemiInj(J) = LocInv(Q)\cap J$ such that $x \in U_sJ$ and $U_sJ$ is a classical order in $U_sQ$ (with respect to a monad of $U_sJ$ which is easily seen to coincide with $Inj(U_sJ)$ since $Q=Soc(Q)$, and therefore $U_sJ$ is a classical order in $U_sQ$). Since $s$ is locally invertible in $Q$,  by  \ref{localyinvert} we can consider the idempotent $e=P(s)\in Q$. With the notation introduced above, we have
$U_sJ \subseteq J_2(e) \subseteq  Q_2(e) = U_sQ$, and since $U_sJ$ is a classical order in $U_sQ = Q_2(e)$, $J_2(e)$ is also a classical order in $Q_2(e)$ by \cite[Corollary 2.3]{fgm}.

Clearly $x\in U_sQ = U_eQ= Q_2(e)$ implies $ x\in J_2(e)$, and since $J_2(e) \subseteq J$, the containment $J_2(e)_x\subseteq J_x$ is clear. From the obvious containment $Q_2(e)_x \subseteq Q_x$,  making use of the identification $Q_2(e) = Q_e$  (see \cite[Example 1.12]{Subquot}) yields  $Q_2(e)_x = (Q_e)_{x+Ker e}= Q_{U_ex}= Q_x$ (see \cite[Lemma 0.5]{pi-i}).

Since as we have mentioned before $J_2(e)$ is a classical order in $Q_2(e)$, by \ref{localofclassicalordersatsocle} $J_2(e)_x$ is a classical order in $Q_2(e)_x$, and since $J_2(e)_x \subseteq J_x\subseteq Q_x= Q_2(e)_x$ (with the identifications mentioned before), from \cite[Corollary 2.3]{fgm} we obtain that $J_x$ is a classical order in $Q_x$, as asserted in LOS3.

We next address the reciprocal, so we assume that $J\subseteq Q$ is a nondegenerate subalgebra of a Jordan algebra $Q = Soc(Q)$, and that properties LOS1, LOS2 and LOS3 hold. It is clear that to prove that $J$ is a local order on $Q$ it suffices to prove that for any $q\in Q$ there exists $x\in SemiInj(J)$ such that $q \in U_xQ$ and $U_xJ$ is a classical order in $U_xQ$. As shown in \ref{quotlocalsemiinj}, this is equivalent to proving that $J_{x^2}$ is a classical order in $Q_{x^2}$, which obviously follows from LOS3.\end{proof}

We have shown in \ref{intersection inner ideals} that if $J$ is an order in a Jordan algebra with dcc on principal inner ideals (that is, such that $Soc(Q)=Q$), and $J$ is nondegenerate, then $Q$ is also nondegenerate. We next prove the reciprocal of that assertion:

\begin{lema}\label{lift degeneracy}Let $J$ be a local order in a Jordan algebra $Q$ with
$Soc(Q) =Q$. If $Q$ is nondegenerate then $J$ is nondegenerate.\end{lema}
\begin{proof}Suppose that $z\in J$ is an absolute zero divisor of $J$. Since $J$ is a local order in $Q = Soc(Q)$, there exists an idempotent $e = P(s) \in Q$ for some $s\in J\cap LocInv(Q)$ such that $z\in U_sQ = Q_2(e)$. Now, since $J$ is a local order in $Q$, arguing as in \ref{whylocalarelocal}, $J_2(e)$ is a classical order in $Q_2(e)$, and moreover, the algebra $Q_2(e)$ has finite capacity since $Q=Soc(Q)$, and is nondegenerate since so is $Q$ by hypothesis. Now, the part of the proof of $(i)\Rightarrow(ii)$ of \cite[Theorem 9.3]{fgm} that asserts that a classical order in a nondegenerate artinian Jordan algebra is itself nondegenerate applies here verbatim to give that $J_2(e)$ is nondegenerate, taking into account that a nondegenerate algebra of finite capacity is a direct sum of a finite number of simple algebras with finite capacity \cite[Theorem 6.4.1]{jac-struc}. Now $U_zJ_2(e) \subseteq U_zJ =0$, hence $z$ is an absolute divisor in the nonegenerate algebra $J_2(e)$, and therefore $z=0$.\end{proof}

\begin{lema}\label{quotwithsocleislocal}Let $J$ be a nondegenerate Jordan algebra, and let $Q\supseteq J$ be an algebra of quotients of $J$. If $Q= Soc(Q)$, then $SemiInj(J) \subseteq
LocInv(Q)$. In particular, if $J$ is a nondegenerate algebra with dcc on inner ideals,
then $SemiInj(J) = LocInv(J)$.\end{lema}
\begin{proof}Since $SemiInj(J) \subseteq SemiInj(Q)$ by \ref{semiinjgotoquotients}, we can assume that $J=Q = Soc(Q)$. Take then $x\in SemiInj(J)$, and let $x = x_2+x_0$ with $x_i \in J_i(e)$ be its Fitting decomposition with respect to the corresponding idempotent $e$, which exists since $J$ has dcc on principal inner ideals (see \cite[Theorem 1]{loos-fitting}). Our aim is to prove that $x_0 =0$. If, on the contrary $x_0 \ne 0$, the element $x_0$ being nilpotent implies that there exists $n\ge 2$ such that $x_0^n =0$ and $x^{n-1}\ne 0$. Since $J$ is nondegenerate, so is  $J_0(e)$ and we can choose $z_0\in J_0(e)$ such that $U_{x_0^{n-1}}z_0 \ne 0$. An easy induction using the multiplication properties of the Peirce decomposition shows that for any $y_0 \in J_0(e)$ and any $m\ge 0$, $U^m_xy_0 = U_{x_0^m}y_0$. Since $x_0^n =0$, we have  $U_x^2U_{x^{n-2}}= U_x^nz_0 = U_{x_0^n}y_0$, and since $x$ is semi-injective, this implies that $0 = U_xU_x^{n-2}y_0 = U_x^{n-1}y_0 = U_{x_0^{n-1}}y_0$, which contradicts the choice of $y_0$. Therefore $x = x_2$ is invertible in $J_2(e)$, hence $x \in LocInv(J)$.
\end{proof}

\begin{lema}\label{LOofsocle}Let $Q$ be a nondegenerate Jordan algebra with $Q = Soc(Q)$. If $Q$ is a local order in a Jordan algebra $\tilde Q$ with $\tilde Q = Soc(\tilde Q)$, then $\tilde Q = Q$.\end{lema}
\begin{proof}Take $\tilde q\in \tilde Q$. By \ref{whylocalarelocal}, there exists an element $x\in SemiInj(Q)$ such that $\tilde q\in U_x\tilde Q$, and $U_xQ$ is a classical order in $U_x\tilde Q$. Since $Q= Soc(Q)$, by \ref{localsemiinject} $U_xQ\cong Q_{x^2}$  is artinian, hence $U_xQ = U_x\tilde Q$, and therefore $\tilde q\in U_x\tilde Q = U_xQ\subseteq Q$. Since this holds for an arbitrary $\tilde q\in \tilde Q$, we get $Q = \tilde Q$.
\end{proof}

\begin{lema}\label{ann_lc}Let $J$ be a nondegenerate Jordan algebra wich is a local order in a Jordan algebra $\tilde J$ such that $Soc(\tilde J) =\tilde J$. Then
\begin{enumerate}\item[(i)]$ann_J(X) = ann_{\tilde J}(X)\cap J$ for any $X\subseteq J$.
\item[(ii)]For any subsets $X, Y\subseteq J$, $ann_J(X)\subseteq ann_J(Y)$ if and only if $ann_{\tilde J}(X)\subseteq
ann_{\tilde J}(Y)$, and $ann_J(X)\neq0$ if and only if $ann_{\tilde J}(X)\neq0$.
\item[(iii)] $J$ satisfies acc on annihilators of its elements.
\end{enumerate}
\end{lema}

\begin{proof} (i) Clearly, we can assume that $X =\{x\}$ has just one element. The containment $ann_{\tilde J}(x)\cap J\subseteq ann_J(x)$ being obvious, we only have to prove the reverse containment, which will follow from the containment $ann_J(x)\subseteq  ann_{\tilde J}(x)$.

Take then $z\in ann_J(x)$. Since $U_zx=U_xz=0$, according to \cite[Remark 1.7]{mc-ann}, we only need to prove $U_zU_x\tilde J=0= \{z,x,\tilde J\}$. Take an arbitrary $q \in \tilde J$. Since $\tilde J= Soc(\tilde J)$ and $\tilde J$ is nondegenerate by \ref{intersection inner ideals}, we can apply \ref{litoff_QJA} to find an idempotent $e\in \tilde J$ such that $x, z, q \in \tilde J_2(e)$.

Now, by \ref{InTEx_idempotents}, there exists $s\in SemiInj(J)$ such that $e=P(s)$, and since $J_{s^2}$ is a classical order in $\tilde J_{s^2}$ by LOS3, arguing as in \ref{quotlocalsemiinj}, we get that $U_sJ$ is a classical order in $U_s\tilde J$. Therefore, applying \cite[Corollary 2.3]{fgm} to  the containment $U_sJ\subseteq J_2(e) \subseteq \tilde J_2(e) = U_s\tilde J$, we obtain that $J_2(e)$ is a classical order in $\tilde J_2(e)$. Then $ann_{J_2(e)}(x) = ann_{\tilde J_2(e)}(X)\cap J_2(e)$ by \cite[Proposition 2.8]{fgm}. Since $x,z,q\in J_2(e)$, this implies that $U_zU_xq=0= \{x,z,q\}$, as desired.

(ii) That $ann_{\tilde J}(X)\subseteq ann_{\tilde J}(Y)$ implies $ann_J(X)\subseteq
ann_J(Y)$ readily follows from (i). For the reciprocal, assume that $ann_J(X)\subseteq
ann_J(Y)$ and suppose that $ann_{\tilde J}(X)\not\subseteq ann_{\tilde J}(Y)$. Since $\tilde J$ satisfies the dcc on principal inner ideals, we can choose an inner ideal $U_a\tilde J$, minimal among the inner ideals generated by the elements of $ann_{\tilde J}(X)$ which do not belong to $ann_{\tilde J}(Y)$. Since $a\in Soc(\tilde J) = \tilde J$, we can complete $a$ to an idempotent $e=(e^+,e^-)$ with $a= e^+$ of the pair $V(\tilde J) = (\tilde J,\tilde J)$. We then have $U_a\tilde J = Q_{e^+}\tilde J \supseteq Q_{e^+}Q_{e^-}\tilde J =\tilde J_2(e)\supseteq U_a\tilde J$, and $U_a\tilde J$ is nondegenerate and has finite capacity. Then \cite{loos-finiteness} implies that the pair $V(\tilde J_2(e))$ contains a strong frame $F = \{e_1,\dots, e_m\}$, hence $f = \sum F= e_1+\cdots + e_m$ is complete in $V(\tilde J_2(e))$, and therefore $f^+$ is invertible in the algebra $V(\tilde J_2(e))^+ = \tilde J_2(e)$. Then $V(\tilde J)^+_2(f)= V(\tilde J)^+_2 = U_a\tilde J$. Now, if $e_i^+ \in ann_{\tilde J}(Y)$ for all $i$, then $f^+ \in ann_{\tilde J}(U)$, hence $U_a\tilde J = V(\tilde J)_2^+(f) \subseteq ann_{\tilde J}(a)$, which contradicts the choice of $a$. Thus we can assume that $e_1^+ \not\in ann_{\tilde J}(Y)$, and since $e_1^+ \in U_a\tilde J$, the minimality of $U_a\tilde J$ implies that $U_a\tilde J = U_{e_1^+}\tilde J$. Thus, since $e_1$ is a division idempotent, we can assume that $U_a\tilde J$ is a minimal inner ideal. We now apply \ref{intersection inner ideals}(i) to find an element $0\ne d\in U_aJ\cap J \subseteq ann_{\tilde J}(X)\cap J = ann_J(X)\subseteq ann_J(Y) = ann_{\tilde J}(Y)\cap J \subseteq ann_{\tilde J}(Y)$.  But this implies that $a \in U_a\tilde J=U_d\tilde J \subseteq ann_{\tilde J}(Y)$, a contradiction.

If $ann_{\tilde J}(X) = 0$, then, by (i) $ann_J(X) = ann_{\tilde J}(X)\cap J =0$. The reciprocal is straightfoward  from \ref{intersection inner ideals}.

(iii) Since $\tilde J$ satisfies dcc on principal inner ideals, by \cite{fgl} it satisfies  acc on annihilators of single elements, and so does $J$ by (ii).
\end{proof}

\begin{teorema}\label{localordersareaq}Let the Jordan algebra $J$ be a local order in the algebra $Q$ satisfying $Soc(Q) = Q$, then $Q$ is a general algebra of quotients of $J$.
\end{teorema}
\begin{proof}Since $Q=Soc(Q)$, all the local algebras $Q_q$ for $q \in Q$ have finite capacity, and since $J_x$ is a classical order in the algebra $Q_x$ by LOS3, it is  LC, hence $x\in LC(J)$, that is $J = LC(J)$. Then dense inner ideals are essential inner ideals, so we only need to prove that for any $q\in Q$ the inner ideal ${\cal D}_J(q)$ is essential, and $U_q{\cal D}_J(q)\ne 0$. We begin by proving that the first statement implies the second one. Indeed, if $L$ is an essential inner ideal of $J$ and $U_qL=0$, then any $z\in U_qQ\cap J$ has $U_zL=0$, and therefore it is an essential zero divisor.
Since $J$ is strongly nonsingular, this implies $U_qQ\cap J = 0$, which contradicts \ref{intersection inner ideals}(i).

Now, to prove the first assertion we have to prove that ${\cal D}_J(q)\cap U_aJ\ne 0$ for any nonzero $a \in J$. By \ref{litoff_QJA} and \ref{InTEx_idempotents}, there exists an idempotent $e = P(s)$ for a semiinjective element $x\in J$, such that $a,q \in Q_2(e)$. As we have already noticed before, since $U_sJ$ is a classical order in $U_sQ$, then $J_2(e)$ is a classical order in $Q_2(e)$ (\ref{almostpeirce}, \ref{whylocalarelocal}). Then ${\cal D}_{J_2(e)}(q)$ is an essential inner ideal of $J_2(e)$, hence there exists an element $d\in {\cal D}_{J_2(e)}(q)\cap U_aJ_2(e)$. Note that $J_2(e)$ is strongly nonsingular, since it is LC and \ref{lckillsstrongzeta} applies, and there exists an element $y\in {\cal D}_{J_2(e)}(q)$ such that $0 \ne z=U_dy\in {\cal D}_{J_2(e)}(a)\cap U_aJ$.  Let us see that $z\in {\cal D}_J(q)$.

Since $z \in {\cal D}_{J_2(e)}(a) \subseteq J_2(e)$,  we have $U_zq$ and  $U_qz$ belong to $J_2(e)\subseteq J$, so (Di) and (Dii) of \ref{denominadores} hold. Also, $z\circ q \in J_2(e) \subseteq J$, and if $c \in J$, we get $\{z,a,c\} = \{U_dy, a,c\} = \{d, \{y, d, a\}, c\}- \{U_da,y,c\} \in \{J,J_2(e),J\}+\{J_2(e),J,J\} \subseteq J$, so (Div) of \ref{denominadores} holds. Finally, $U_aU_zc = U_aU_{U_dy}c= U_aU_d(U_yJ) \subseteq U_aU_dJ_2(e) \subseteq J_2(e) \subseteq J$ because $U_yJ = U_{U_ey}J = U_eU_yU_eJ \subseteq Q_2(e)$ gives $U_yJ\subseteq Q_2(e)\cap J = J_2(e)$. Therefore $U_aU_zJ\subseteq J$, and (Diii) of \ref{denominadores} also holds, and this implies that $0\ne z\in {\cal D}_J(q)\cap U_aJ$, and thus, the essentiality, hence the density, of ${\cal D}_J(q)$.
\end{proof}

%
%

\section{Local Goldie-like conditions and local orders.}

In this section we give a Goldie-like characterization  of quadratic Jordan algebras that safisfy
local Goldie conditions extending the results given in \cite{fg1,fg2}. We first turn our attention to the Local Goldie Conditions as introduced in \cite{fg1}.

\begin{apartado}\label{local conditions} An element $x$ in a Jordan algebra $J$ has \emph{finite uniform} (or \emph{Goldie}) \emph{dimension}
if $U_xJ$ does not contain infinite direct sums of inner ideals of
$J$ (in the sense of \cite[p. 425]{fgm}).

Every element in a simple Jordan algebra with
dcc on principal inner ideals which is not a nonartinian quadratic
factor has finite uniform dimension.
 \end{apartado}

\begin{lema}\label{ii_damour}
Let $J$ be a nondegenerate Jordan algebra. An element $x$ in $J$ has finite uniform dimension if and only if the local algebra $J_x$ has finite uniform dimension.
\end{lema}
\begin{proof}  By   \cite[Proposition 2.4]{am-local} it suffices to note that  for any element $x$ in a
nondegenerate Jordan algebra $J$ there is a bijective order preserving correspondence between the set of inner ideals of the local algebra $J_x$ and the set of inner ideals of $J$ contained in $U_xJ$.\end{proof}

\begin{lema}\label{quot_intersection}    Let $J$ be a nondegenerate Jordan algebra and let $Q$ be an algebra of quotients of $J$. Then any nonzero ideal $I$ of $Q$ is an algebra of quotients of $I\cap J$.
\end{lema}
\begin{proof}
Note  that, by \ref{intersection inner ideals},  $I\cap J$ is a nonzero ideal of $J$, and moreover $I\cap J$ is nondegenerate as an algebra \cite{mc-inh}.

Let $q\in  I$.
Clearly ${\cal D}_J(q)\cap (I\cap J)={\cal D}_J(q)\cap I\subseteq {\cal D}_{I\cap J}(q)$, and for any $a\in I\cap J$ we have $({\cal D}_J(q):a)_L\cap I\subseteq ({\cal D}_{I\cap J}(q):a)_L$.
Thus for any $a,b\in I\cap J$ it holds that $({\cal D}_J(q):a)_L\cap ({\cal D}_J(q):b)_L\cap I\subseteq ({\cal D}_{I\cap J}(q):a)_L\cap ({\cal D}_{I\cap J}(q):b)_L$.

By \cite[Proposition 1.9]{densos}, to prove the density of ${\cal D}_{I\cap J}(q)$ in $I\cap J$ it suffices to prove that
 $U_c\big(({\cal D}_{I\cap J}(q):a)_L\cap ({\cal D}_{I\cap J}(q):b)_L\big)=0$ for any $a,b,c\in I\cap J$ implies  $c=0$.

Let $a,b,c\in I\cap J$ and assume
$U_c\big(({\cal D}_{I\cap J}(q):a)_L\cap ({\cal D}_{I\cap J}(q):b)_L\big)=~0$. Then we have
$U_c\big(({\cal D}_J(q):a)_L\cap ({\cal D}_J(q):b)_L\big)=0$. Let $K=({\cal D}_J(q):a)_L\cap ({\cal D}_J(q):b)_L$ and take $x\in K$. Then, since  $K$ is an inner ideal of $J$, $U_{U_cx}I=U_cU_xU_cI\subseteq U_cU_KI\subseteq U_c(K\cap I)=0$. Hence $U_cK\subseteq ann_Q(I)$.
But $c\in I\cap J$, hence $U_cK\subseteq ann_Q(I)\cap I=0$. Thus $U_cK=0$, which implies  $c=0$ by the density of ${\cal D}_J(q)$ in $J$ and \cite[Proposition 1.9]{densos}. Hence  ${\cal D}_{I\cap J}(q)$ is a dense inner ideal of $I\cap J$.

Finally suppose that $U_q{\cal D}_{I\cap J}(q)=0$. Then we have $U_{U_q{\cal D}_J(q)}{\cal D}_{I\cap J}(q)=U_qU_{{\cal D}_J(q)}U_q{\cal D}_{I\cap J}(q)
=~0$, and since $U_q{\cal D}_J(q)\subseteq I\cap J$, the density of ${\cal D}_{I\cap J}(q) $ in $I\cap J$ implies that $U_q{\cal D}_J(q)=0$. Hence  $q=0$ because $Q$ is an algebra of quotients of $J$.
\end{proof}

\begin{teorema}\label{lc_local_order_socle}    Let $J$ be a nondegenerate Jordan algebra with maximal algebra of quotients $Q_{max}(J)$. Then $LC(J)$ is a local order in $Soc(Q_{max}(J))$.
\end{teorema}
\begin{proof}
Let us denote $Q = Q_{max}(J)$. Since  $LC(J)=J\cap Soc(Q)$ by \ref{caso general}, we can assume that $J=LC(J)$ by \ref{quot_intersection}, so that $J$  has an algebra of quotients $Q=Soc(Q)$ with dcc on principal inner ideals. By \ref{quotwithsocleislocal}, $SemiInj(J)\subseteq LocInv(Q)$ hence $SemiInj(J) = LocInv(Q)\cap J$ which is LOS1. Take now an element  $q\in Q = Soc(Q)$, using \ref{litoff_QJA} we obtain an idempotent $e\in Q$ such that $q\in U_eQ$, and using  \ref{InTEx_idempotents}, we can find an element $s\in SemiInj(J)$ such that $e = P(s)$. Thus $q \in U_sQ$, and LOS2 holds. Finally, it suffices to prove LOS3, for any $x\in J$, $J_x$ is a classical order in $Q_x$, and this follows as usual from  \cite[Proposition 2.10]{fgm}
applied to the fact that $Q_x$ is an algebra of quotients of $J_x$ (\ref{local-q}), and $Q_x$ has finite capacity since $x\in Soc(Q)$.
\end{proof}

As a straightforward application of  \ref{lc_local_order_socle} we get:

\begin{teorema}\label{lc_local_order_dcc}    Let $J$ be a Jordan algebra. Then the follo\-wing conditions are equivalent:
\begin{enumerate}
\item[(i)] $J$ is a local order in a nondegenerate Jordan algebra $Q$ with dcc on principal inner ideals.
\item[(ii)] $J$ is nondegenerate, and $J=LC(J)$ (that is $J_x$ is LC for all $x\in J$).
\end{enumerate}
In this case, $J$ is strongly prime if and only if $Q$ is simple.
\end{teorema}
\begin{proof}
(i)$\Rightarrow$(ii) By \ref{lift degeneracy}, $J$ is nondegenerate. Now, let $a\in J$. Since $J$ is a local order in $Q$, by LOS3 of \ref{whylocalarelocal}, $J_a$ is a classical order in $Q_a$, and since $a \in Soc(Q) = Q$, $Q_a$ has finite capacity. Therefore $a \in LC(J)$.

(ii)$\Rightarrow$(i)  Assume now that $J$ is nondegenerate and  satisfies $J=LC(J)$. Then by \ref{lc_local_order_socle}, $J$ is a local order in $Soc(Q_{max}(J))$ which is nondegenerate and has dcc on principal inner ideals.

The last assertion easily follows from the relation between the annihilators of $J$ and $Q$ proved in \ref{ann_lc}, and the relation between annihilator ideals of a nondegenerate algebra and its algebras of quotients \cite [Lemma 5.3 and Lemma 5.4]{densos}.
\end{proof}

\begin{teorema}\label{main local goldie}
For a Jordan algebra $J$  the following conditions are equiva\-lent:
\begin{enumerate}
\item[(1)] $J$ is a local order in a nondegenerate locally artinian Jordan algebra $Q$.
\item[(2)]  $J$ is nondegenerate, satisfies the acc on annihilators of its elements, and every element $x\in J$ has finite uniform dimension.
\item[(3)] $J$ is nondegenerate and $J_x$ is Goldie for all nonzero element $x\in J$.
\end{enumerate}
Moreover, $J$ is strongly prime if and only if $Q$ is simple.
\end{teorema}\begin{proof}
(1)$\Rightarrow$(2) Since $J$ is a local order in a nondegenerate $Q$ with dcc on principal inner ideals, the nondegeneracy of $J$ follows from \ref{lift degeneracy}, while acc on annihilators of its elements is \ref{ann_lc}(iii). Finally by \ref{ii_damour}, the fact that all $x\in J$ have finite uniform dimension is equivalent to each $J_x$ having finite uniform dimension, and since by LOS3 $J_x$ is a classical order in the artinian algebra $Q_x$, it follows from \ref{goldieThm} that $J_x$ has finite uniform dimension.

(2)$\Rightarrow$(3) We already have that $J$ is nondegenerate by hypothesis. We prove that for all $x\in J$, the local algebra $J_x$ is Goldie by appealing to characterization (iii) of \ref{goldieThm}. Since we are already assuming that $J_x$ has finite uniform dimension, it suffices to prove that each nonzero ideal $\bar I $ of $J_x$ contains a uniform element. Set $I = \{ y \in J\mid y+Ker\,x \in \bar I\}$, the preimage of $\bar I$ by the projection $J\rightarrow J_x$. Since $J$ has the acc on annihilators of elements every nonzero inner ideal contains a uniform element, in particular we can choose a uniform element $z_0\in U_xI$ since $U_xI$ is a nonzero inner ideal. Clearly $z_0 = U_xy_0$ for some $y_0 \in I$, and since $z_0$ is uniform, $(J_x)_{\overline{y_0}} = U_{U_xy_0} = J_{z_0}$ is a Jordan domain, hence $\overline{z_0} \in \bar I$ is uniform, and therefore $J_x$ is Goldie.

(3)$\Rightarrow$(1) Since every local algebra $J_x$ of $J$ is Goldie, $J$ has $J = LC(J)$. From  \ref{lc_local_order_socle} we know that $J$ is a local order in $Q= Soc(Q_{max}(J))$, so it suffices to note that $Q$ is locally artinian since $Q_x$ is has finite capacity by \ref{lc_local_order_socle}, and is a  classical algebra of quotients of the Goldie algebra $J_x$ by LOS3 of \ref{whylocalarelocal}.

The last assertion follows directly as a particular case of the last assertion of \ref{lc_local_order_dcc}.
\end{proof}

We finally prove that algebras with dcc on principal inner ideals in which a given algebra is a local order are unique.

\begin{proposicion}Let $Q_1$, $Q_2$ be Jordan algebras and let $J$ be a local order in both $Q_1$ and $Q_2$. If $J$ is nondegenerate, $Soc(Q_1) = Q_1$, and $Soc(Q_2) = Q_2$, then there exists a unique isomorphism $\alpha: Q_1\rightarrow Q_2$ that extends the identity mapping $J\rightarrow J$.\end{proposicion}
\begin{proof}By \ref{localordersareaq}, both $Q_1$ and $Q_2$ are general algebras of quotients of $J$. Then $Q_{max}(Q_i) = Q_{max}(J)$ for $i=1,2$, and by \cite[Lemma 2.12]{densos}, there is a unique isomorphism $\beta: Q_{max}(Q_1) \rightarrow Q_{max}(Q_2)$ which restricts to the identity on $J$. Clearly $\beta(Soc(Q_{max}(Q_1))\subseteq Soc(Q_{max}(Q_2))$. Switching the roles of $Q_1$ and $Q_2$ we obtain a unique isomorphism $\gamma: Q_{max}(Q_2) \rightarrow Q_{max}(Q_1)$ which restricts to the identity on $J$, and again it is clear that $\gamma(Soc(Q_{max}(Q_2))\subseteq Soc(Q_{max}(Q_1))$. Therefore, from the uniqueness of the isomorphisms we have $\gamma= \beta^{-1}$, and $\beta$ restricts to an isomorphism $Soc(Q_{max}(Q_1))\cong Soc(Q_{max}(Q_2))$ which in turn, restricts to the identity on $J$.  This isomorphism obviously induces an isomorphism $LC(Soc(Q_{max}(Q_1)))\cong LC(Soc(Q_{max}(Q_2)))$. Now, since $Q_i= Soc(Q_i)$ implies $Q_i = LC(Q_i)$, $Q_i$ is  a local order in $LC(Soc(Q_{max}(Q_i)))$ by \ref{lc_local_order_socle}, and we get $Q_i = LC(Soc(Q_{max}(Q_i)))$ by \ref{LOofsocle}, and thus there is a unique isomorphism $Q_1\cong Q_2$ which restricts to the identity on $J$.\end{proof}

%
%
 \addcontentsline{toc}{chapter}{\protect\numberline{}  Bibliograf\'{\i}a.}



\begin{thebibliography}{11111111}

  \bibitem[A]{anh}  P. N. Anh, Simple Jordan algebras with minimal ideals,    Comm. Algebra 14 (3)
(1986), 489-492.
\bibitem[AM1]{anhm1} P. M. \'Anh, L. M\'arki, A general theory of Fountain-Gould quotient rings, Math. Slovaca, 44 (1994), 225-235.
\bibitem[AM2]{anhm2} P. M. \'Anh, L. M\'arki, Orders in primitive rings with nonzero socle and Posner's theorem, Comm. Algebra 24 (1996), 289-294.
  \bibitem[ACM]{acm}  J. A. Anquela, T. Cort\'{e}s, F. Montaner,  Local inheritance in Jordan algebras,    Arch. Math.  64
(1995), 393-401.

  \bibitem[DM]{am-local} A. D'Amour, K. McCrimmon,  The local algebras of Jordan systems,
  J. Algebra  177
(1995), 199-239.

\bibitem[BM]{bomc} J. Bowling, K. McCrimmon, Outer fractions in quadratic Jordan algebras, J. Algebra 312 (2007), 56-73.
   \bibitem[FG1]{fg1} A. Fern\'{a}ndez L\'{o}pez, E. Garc\'{\i}a Rus,  Prime Jordan Algebras sa\-tisfying Local
    Goldie conditions,  J. Algebra  174
 (1995), 1024-1048.
     \bibitem[FG2]{fg2} A. Fern\'{a}ndez L\'{o}pez, E. Garc\'{\i}a Rus,  Nondegenerate Jordan Algebras satisfying Local
    Goldie conditions,  J. Algebra   182
 (1996), 52-59.
  \bibitem[FGL]{fgl} A. Fern\'{a}ndez L\'{o}pez, E. Garc\'{\i}a Rus, O. Loos, Annihilators of elements of the socle of a Jordan algebra, Comm. Algebra 22 (1994), 2837-2844.

  \bibitem[FGM]{fgm} A. Fern\'{a}ndez L\'{o}pez, E. Garc\'{\i}a Rus, F. Montaner, Goldie Theory for Jordan Algebras,  J. Algebra
   248
 (2002), 397-471.
   \bibitem[FGSS]{fgss} A. Fern\'{a}ndez L\'{o}pez, E. Garc\'{\i}a Rus, E. S\'{a}nchez Campos, M. Siles Molina,
 Strong regularity and generalized inverses in   Jordan  systems,  Comm. Algebra
 20 (7)
 (1992), 1917-1936.
 \bibitem[FT]{ft} A. Fern\'{a}ndez L\'{o}pez, M. I. Toc\'{o}n,
   Strongly prime Jordan pairs with nonzero socle,  Manuscripta
   Math.
 111  (2003), 321-340.
  \bibitem[FoGo1]{fogo1} J. Fountain, V. Gould, Orders in rings without identity, Comm. Algebra 18 (1990), 3085-3110.
  \bibitem[FoGo2]{fogo2} J. Fountain, V. Gould, Orders in regular rings with minimal condition for principal right ideals, Comm. Algebra 19 (1991), 1501-1527.
  \bibitem[FoGo3]{fogo3} J. Fountain, V. Gould, Straight left orders in rings, Quat. J. Math. Oxford Ser. 43 (1992), 303-311.
  \bibitem[G1]{goldie1} A.W. Goldie, The structure of prime rings
 under ascending chain conditions,  Proc. London Math. Soc.
8  (1958), 589-608.
 \bibitem[G2]{goldie2} A.W. Goldie, Semiprime rings with maximum
 condition,
  Proc. London Math. Soc.
10  (1960), 201-220.
  \bibitem[J1]{jac-struc-rep} N. Jacobson,  Structure and Representations of Jordan Algebras, Amer. Math. Soc. Colloq. Publ., vol. 39, American Mathematical Society, Providence, RI, 1968.
   \bibitem[J2]{jac-struc} N. Jacobson,  Structure  Theory of Jordan Algebras,  The University  of Arkansas Lecture
Notes in Mathematics, Vol.~5, University of Arkansas, Fayetteville,
1981.
 \bibitem[L]{lanning} S. Lanning, The maximal symmetric ring of
quotients,
    J. Algebra   179
 (1996), 47-91.
\bibitem[Lo1]{loos-jp} O. Loos,  Jordan Pairs,   Lecture
Notes in Mathematics, Vol.~460, Springer-Verlag, New York, 1975.
\bibitem[Lo2]{loos-socle} O. Loos, On the socle of a  Jordan pair,   Collect. Math. 40 (1989), 109-125.
\bibitem[Lo3]{loos-fitting} O. Loos, Fitting decomposition in Jordan systems, J. Algebra 136 (1991), 92-102
\bibitem[Lo4]{loos-finiteness} O. Loos, Finiteness conditions in Jordan pairs,   Math.
Z.  206
  (1991), 577-587.
  \bibitem[Lo5]{loos-diagonal} O. Loos, Diagonalization in Jordan pairs,   J. Algebra 143
  (1991), 252-268.
  \bibitem[LN]{Subquot}O. Loos, E. Neher, Complementation of inner ideals in Jordan pairs, J. Algebra 166 (1994), 255-295.

  \bibitem[Ma]{martinez} C. Mart{\'\i}nez,  The ring of fractions of a Jordan algebra,
  J. Algebra  237
 (2001), 798-812.
 \bibitem[Mc1]{mc-in} K. McCrimmon, Inner ideals in Quadratic Jordan Algebras,
  Trans. Amer. Math. Soc  159
  (1971), 445-468.
  \bibitem[Mc2]{mc-ann}K. McCrimmon, The Zelmanov annihilator and nilpotence of the of the nil radical in quadratic Jordan algebras with chain conditions, J. Algebra 67 (1980), 230-253.
 \bibitem[Mc3]{mc-inh} K. McCrimmon, Strong Prime Inheritance in Jordan Systems,  Algebras, Groups and Geometries  1
  (1984), 217-234.
\bibitem[McZ]{mcz}   K. McCrimmon, E. Zelmanov, The Structure of strongly prime Quadratic Jordan Algebras,
  Adv. Math.  69  (2)
 (1988), 133-222.
\bibitem[Mo1]{pi-i}   F. Montaner, Local PI-Theory of Jordan Systems,
  J. Algebra  216
 (1999), 302-327.
\bibitem[Mo2]{pi-ii}   F. Montaner, Local PI-Theory of Jordan Systems II,
 J. Algebra  241
 (2001), 473-514.
 \bibitem[Mo3]{densos}   F. Montaner, Maximal algebras of quotients
 of Jordan algebras, J. Algebra
 323
 (2010), 2638-2670.
 \bibitem[MoP]{esenciales}   F. Montaner, I. Paniello, Algebras of
 quotients of nonsingular Jordan algebras,   J. Algebra
 312
 (2007), 963-984.
  \bibitem[MoT1]{lc}   F. Montaner, M. I. Toc\'{o}n, Local Lesieur-Croisot
  theory of Jordan algebras,   J. Algebra  301
 (2006), 256-273.
 \bibitem[MoT2]{lc-pr} F. Montaner, M. I. Toc\'{o}n, The ideal of Lesieur-Croisot elements of a Jordan algebra. II,  Algebras, representations and applications, 199-203, Contemp. Math., 483, Amer. Math. Soc., Providence, RI, 2009.
 \bibitem[OR]{orsocle} J. M. Osborn, M. L. Racine, Jordan rings with nonzero socle. Trans. Amer. Math, Soc. 251 (1979), 375-387.
 \bibitem[P]{tesis} I. Paniello, Identidades polin\'{o}micas y \'{a}lgebras de cocientes en sistemas de Jordan. Doctoral Dissertation, Universidad de Zaragoza, 2004.
 \bibitem[Ro]{anillos}  L. H. Rowen,  Ring Theory, Vol. I, Academic
Press, New York, 1988.
 \bibitem[St]{st}  B. Stenstr\"{o}m, Rings of Quotients, Springer-Verlag,
New York, 1975.
 \bibitem[Z1]{z-goldie-1} E. Zelmanov, Goldie's theorem for Jordan algebras, Siberian Math. J. 28 (1987), 44-52.
 \bibitem[Z2]{z-goldie-2} E. Zelmanov, Goldie's theorem for Jordan algebras II, Siberian Math. J. 29 (1988), 68-74.

\end{thebibliography}
\end{document}